\documentclass[journal]{IEEEtran}

\ifCLASSINFOpdf
 
\else
 
\fi

\usepackage{tikz}

\usepackage{subcaption}
\usepackage{caption}

\usepackage{url}  
\usepackage[numbers]{natbib}
\usepackage{amsmath}
\usepackage{amsfonts}
\newcommand{\V}{\mathcal{V}} 
\newcommand{\E}{\mathcal{E}} 
\newcommand{\Nb}{\mathcal{N}} 
\newcommand{\W}{\mathbf{W}} 
\newcommand{\D}{\mathbf{D}} 
\newcommand{\A}{\mathbf{A}} 
\newcommand{\X}{\mathbf{X}} 
\DeclareMathOperator{\PE}    {PE}
\DeclareMathOperator{\MIX}    {MIX}
\DeclareMathOperator{\PEG}    {PE_G}

\newcommand{\DG}{ \overrightarrow{G}} 
\newcommand{\R}{\mathbb{R}} 
\newcommand{\N}{\mathbb{N}} 
 
\newcommand{\quadtext}[1]{\quad\text{#1}\quad}
\newcommand{\G}{{G}}
\newcommand{\card}[1]{\lvert#1\rvert}

\usepackage{float}

\newcommand{\Sec}[1]{Section~\ref{sec:#1}}

\newcommand{\Eq}[1]{Eq.~\eqref{#1}}

\newcommand{\Fig}[1]{Figure~\ref{fig:#1}}

\newcommand{\Ex}[1]{Example~\ref{ex:#1}}

\newcommand{\Cor}[1]{Corollary~\ref{cor:#1}}

\newcommand{\Prp}[1]{Proposition~\ref{prp:#1}}

\newcommand{\set}[2]{\{ \, #1 \, | \, #2 \, \} }      

\newcommand{\map}[3]{ #1 \colon #2 \longrightarrow #3}

\usepackage{amsthm}

\newtheorem{example}{Example}
\newtheorem{corollary}{Corollary}
\newtheorem{proposition}{Proposition}
\newtheorem*{remark*}{Remark}

\usepackage{notoccite}

\usepackage{amssymb}

\begin{document}

\title{Permutation Entropy for Graph Signals}%

\author{John~Stewart~Fabila-Carrasco,
        Chao~Tan,~\IEEEmembership{Senior Member,~IEEE,}
        and~Javier~Escudero,~\IEEEmembership{Senior Member,~IEEE}
\thanks{J.S. Fabila-Carrasco and J. Escudero are with School of Engineering, Institute for Digital Communications, University of Edinburgh, West Mains Rd, Edinburgh, EH9 3FB, UK (e-mail: John.Fabila@ed.ac.uk and javier.escudero@ed.ac.uk).}
\thanks{C. Tan is with School of Electrical and Information Engineering, Tianjin University, Tianjin 300072, China (e-mail: tanchao@tju.edu.cn)}%
\thanks{J.S.~Fabila-Carrasco and J.~Escudero were supported by the Leverhulme Trust via a Research Project Grant (RPG-2020-158).}
}

\maketitle

\begin{abstract}
Entropy metrics (for example, permutation entropy) are nonlinear measures of irregularity in time series (one-dimensional data). Some of these entropy metrics can be generalised to data on periodic structures such as a grid or lattice pattern (two-dimensional data) using its symmetry, thus enabling their application to images. However, these metrics have not been developed for signals sampled on irregular domains, defined by a graph.  Here, we define for the first time an entropy metric to analyse signals measured over irregular graphs by generalising permutation entropy, a well-established nonlinear metric based on the comparison of neighbouring values within patterns in a time series. Our algorithm is based on comparing signal values on neighbouring nodes, using the adjacency matrix. We show that this generalisation preserves the properties of classical permutation for time series and the recent permutation entropy for images, and it can be applied to any graph structure with synthetic and real signals. We expect the present work to enable the extension of other nonlinear dynamic approaches to graph signals.
\end{abstract}

\begin{IEEEkeywords}
Graph signal processing, Graph Laplacian, Permutation entropy, Adjacency matrix, Irregularity, Nonlinearity Dynamics, Topology, Entropy metric.
\end{IEEEkeywords}

\IEEEpeerreviewmaketitle

\section{Introduction}

\IEEEPARstart{I}{n} the analysis of time series, entropy is a common tool used to describe the probability distribution of the states of a system. 
Based on this concept, the seminal paper~\cite{Bandt2002} introduced the so-called permutation entropy ($\PE$) as a measure to quantify irregularity (or complexity) in time series, a fundamental challenge in data analysis. This entropy involves calculating permutation patterns, i.e., permutations defined by comparing neighbouring values of the time series. In the last years, $\PE$ has been applied in different field as biomedicine~\cite{Cao2004, Olofsen2008}, physical systems~\cite{Yan2012} and economics~\cite{Zunino2009}. Some variants, modifications and extensions of $\PE$ have been introduced, including: a multiscale step~\cite{Azami2016}; changes targeting signals with noise~\cite{Chen2019}; a variation for detecting heartbeat dynamic~\cite{Bian2012}; the inclusion of a nonlinear mapping to consider the differences between the amplitude values~\cite{Azami2018,Rostaghi2016}; considering time reversibility conditions \cite{Martinez2018,Zanin2018}; and extensions to higher dimensions~\cite{Morel2021}.   

A time series can be considered as a one-dimensional data vector (1D), while an image can be regarded as a two-dimensional regular data set (2D). In the field of image processing, several entropy algorithms have been proposed to quantify the irregularity of images as generalisations of their one-dimensional analogous. Examples include: 2D permutation entropy ~\cite{Morel2021}, 2D sample entropy \cite{Silva2016}, 2D dispersion entropy~\cite{azami2019two}, and 2D distribution entropy~\cite{Azami2017a}. Most of the methods are straightforwardly generalised to higher-dimensional periodic structures. The generalisation comes from the fact that the underlying structure (the lattice graph or grid graph, for example for an image) is a periodic structure. Then, the algorithms \cite{Morel2021,Silva2016,azami2019two,Azami2017a} use the symmetry from the structure to compare the values of the signal. However, thus far, it is unclear how to generalise the two-dimensional methods to a general irregular domain (or graph).

The study of data defined on irregular graphs domains is the main interest of graph signal processing (GSP), an active research area in recent years~\cite{Ortega2018,Stankovic2019}. This is motivated by the fact that, new technological advances have enabled the recording of data from complex systems \cite{Shuman2013}. GSP is immediately useful in applications where measures are distributed on irregular domains. Examples include a network of weather stations, vehicular networks or power grids, among others~\cite{Ortega2018,Stankovic2019,Shuman2013}.
 In some cases, the signal domain is not a set of equidistant time points (time series) or a regular grid (image), and in some cases, the data is not related with the space or time. Graphs can model such data and complex interactions, and these new relations may be included in the data processing techniques. Then, some conventional signal-processing operations can be extended to graphs, such as filtering in the spectral and vertex domain, interpolation, subsampling the data with regarding to the graph~\cite{Stankovic2019,Shuman2013,Shuman2016} and generating surrogate graph signals~\cite{Pirondi2016}.

For a time series, the classical $\PE$ is computed based on the successive values of the time series or neighbouring values. These concepts are equivalents on 1D. However, for a signal on a graph, the concept of successive values is unclear, but we have the notion of neighbouring vertices. This concept is fundamental to generalise the permutation entropy for graphs signals ($\PEG$).
In particular, we will consider time series as a signal function on a 1D-graph (an undirected path) and an image as a signal function on a 2D-graph (a grid).

Of note, the concept of graph entropy has been defined in previous literature \cite{Han2012,Passerini2011}.  However, this definition involves the computation of the Laplacian eigenvalues, its probability distribution and the Shannon entropy. Therefore, it measures the complexity/irregularity of the geometric structure and topology of the graph, but not of the signals on the graph itself. 

Thus, here we introduce a measure of the regularity of a signal over a graph, combining the signal values with the topology of the graph, thus extending entropy algorithms for time series and images to graphs.

\subsection*{Contributions} The main contributions of this article are:
\begin{itemize}
	\item For the first time, the concept of a nonlinear entropy metric -permutation entropy- is extended, from unidimensional time series and two-dimensional images to data residing on the vertices of (irregular) undirected graphs.
	\item We explore how the permutation entropy of graph signals depends on both the signal and the graph. We also give conditions to change the graph while maintaining the entropy of a signal.
	\item We show that our algorithm can also be applied to signals on directed graphs and/or weighted graphs.
	\item We illustrate the application of the permutation entropy on graphs algorithm on well-established benchmark synthetic datasets and on real-world data, showing that it generalises well the behaviour of the unidimensional $\PE$ and the recently introduced two-dimensional permutation entropy.
\end{itemize}

\subsection*{Structure of the article}
The outline of the paper is as follows: \Sec{background} introduces the classical permutation entropy and the notation on graph theory used in the article (including the basic definition of the normalised Laplacian). \Sec{method} presents the main contribution: the permutation entropy for graphs signals, including a version for weighted and directed graphs. In addition, the section presents some examples and study how geometric modification on the graph preserves the entropy values of the signal. \Sec{exp} shows how $\PEG$ applies to real and synthetic signals residing on 1D, 2D and irregular graphs. The conclusions and future lines of research are presented in \Sec{conc} and it concludes the paper.

\section{Background and notation}
\label{sec:background}
In this section, we introduce general background information, including the original permutation entropy (Section~\ref{sub:originalPE}), the definition of a graph and the notion of the normalised Laplacian (Section~\ref{graphtheory}). These definitions will be fundamental to generalise the permutation entropy from a time series to a general graph signal case.  

\subsection{Original permutation entropy}\label{sub:originalPE}
Permutation entropy ($\PE$) measures the irregularity of a time series. The algorithm is based on the comparison of neighbouring values within patterns in the time series~\cite{Bandt2002}. It is a simple, robust method and computationally very fast (as it depends linearly on the number of samples of the signal: $O(N)$). For a time series $\textbf{X}=\left\{x_i\right\}_{i=1}^{N}$, the algorithm to compute $\PE$ is the following \cite{Cao2004}:

\begin{enumerate}
	\item For $2\leq m\in\N$ the \emph{embedding dimension} and $L\in\N$ the \emph{delay time}, the \emph{embedding vector} $\textbf{x}_i^m(L)\in\R^m$ is given by
	\begin{equation}\label{eq:embPE}
		\textbf{x}_i^m(L)=\left( x_{i+jL}\right)_{j=0}^{m-1}=\left(x_i,x_{i+L},\dots, x_{i+(m-1)L}\right) 
	\end{equation}
	for all $1\leq i \leq N-(m-1)L$. For practical purposes, the authors~\cite{Bandt2002} suggest to work with $3\leq m \leq 7$. We consider $L=1$ (unless explicitly stated otherwise).
	\item The $m$ real numbers of the embedding vector $\textbf{x}_i^m(L)$ are associated with natural numbers from $1$ to $m$, and then arranged in increasing order. Then, each embedding  vector $\textbf{x}_i^m(L)$ is assigned to one of the $m!$ permutation (also called possible patterns) denoted by $\pi$.
	
	Formally, the embedding vector \[\textbf{x}_i^m(L)=\left( x_i,x_{i+L},\dots, x_{i+(m-1)L}\right) \] is arranged in the increasing order vector: \[\left( x_{i+(k_1-1)L}\leq x_{i+(k_2-1)L} \leq \dots \leq  x_{i+(k_m-1)L}\right). \] Following the convention in~\cite{Cao2004}, if some values are equal, the order is given by the corresponding $k's$. For example, if $ x_{i+(k_{l1}-1)L}= x_{i+(k_{l2}-1)L}$ and $k_{l1}<k_{l1}$, we write $ x_{i+(k_{l1}-1)L}\leq x_{i+(k_{l2}-1)L}$. This convention does not affect the results~\cite{Cuesta2018}.  In particular, the constant vector $\left( 1,1,\dots, 1\right) $ is mapped onto $\left(1,2,\dots,m \right)$.  Therefore, any embedding vector $\textbf{x}_i^m(L)$ is uniquely mapped onto the vector $(k_1,k_2,\dots,k_m)\in \N^m$. 

	\item  The relative frequency for the distinct permutation $\pi_1,\pi_2,\dots,\pi_k$ where $k\leq m!$ is denoted by $p(\pi_1),p(\pi_2),\dots,p(\pi_k)$. The permutation entropy $\PE$ for the time series $\textbf{X}$ is computed as the Shannon entropy for the $k$ distinct permutations as follows
	\[
	\PE (m, L)=-\sum_{i=1}^{k} p(\pi_i) \ln p(\pi_i)\; .
	\]
\end{enumerate}

It is clear that $0\leq \PE (m, L)\leq \ln(m !)$, then, for convenience, it is normalised by  $\ln(m !)$, then
\[0\leq \frac{\PE (m, L)}{\ln(m !)}\leq 1\; .\]

The simple case is for $m=2$ and $L=1$. Given a time series $\textbf{X}$, the idea of $\PE(m,L)$ is organise the $N-1$ pair of neighbours according to their relative values.
Let $p_1$ be the number of pair of neighbours such that $x_t<x_{t+1}$, represented by the permutation $(1,2)$; and $p_2$ be the number of pair of neighbours such that $x_t>x_{t+1}$, represented by the permutation $(2,1)$. Then, using Shannon's entropy:
\begin{equation*}
	\PE(m,L)=-\dfrac{p_1}{N-1}\log \dfrac{p_1}{N-1}-\dfrac{p_2}{N-1}\log \dfrac{p_2}{N-1}\; .
\end{equation*}

In permutation entropy, the ordering of the values is taken into account, but no the magnitude of changes.

An extension of Permutation Entropy to two-dimensional patterns (images) has very recently been published~\cite{Morel2021}. This two-dimensional algorithm takes rectangular windows across the image and, for each window, vectorises its contents. Then, the steps 2 and 3 are applied.

\subsection{Graphs, graph signals and the normalised Laplacian}\label{graphtheory}
An \emph{undirected graph} $G$ is defined as the triple $G = (\V,\E,\A)$
which consists of a finite set of vertices or nodes $\V=\{1,2,3,\dots, N\}$, an edge set $\E \subset \{(i,j): i,j\in\V\}$  and $\A$
is the corresponding $N \times N$ symmetric adjacency matrix on edges with entries $1 = \A_{i j}= \A_{j i}$ if $(i,j)\in \E$ and $0$ otherwise.

Along this article, we consider graphs containing no multiple edges, loops or isolated vertices, i.e. \emph{simple} graphs.

A \emph{graph signal} is a real function defined on the vertices, i.e.,
$\map{\X}{\V}{\R}$. The graph signal $\X$ can be represented as a $N$-dimensional column vector,  $\X=\left[x_1;x_2;\dots; x_N \right] \in \R^N$ (with the same indexing of the vertices).

It is well-know that the power of the adjacency matrix counts the number of $k$-walks between two vertices, i.e., the entry $(\A^k)_{i,j}$ is equal to the number of walks of length equal to $k$  having the vertex $i$ as start and vertex $j$ as end.

Given a graph  $G = (\V,\E,\A)$, we define a function on the vertices $\map{\deg^k}{\V}{\R}$ given by 
\begin{equation}\label{eq:weightver}
	\deg^k(i):=\sum_{j\in \V}(\A^k) _{ij}=\sum_{j\in \V} (\A^k)_{ji}\; ;
\end{equation}
for $k=1$, we write $\deg^1(i)=\deg(i)$, i.e., the degree of a vertex $i$ is the number of edges that are incident to it. 

Given a vertex $i$, we define $\Nb_k(i)$ as the set of all vertices connected to the vertex $i$ with a walk on $k$ edges, i.e.,
\begin{equation}\label{eq:nb}
	\Nb_k(i):=\set{j\in \V}{\scalebox{.9}[1.0]{ it exists a walk on $k$ edges joining $i$ and $j$} }\;,
\end{equation}
with the convention $\Nb_0(i)=\{i\}$ and $\Nb_1(i)=\Nb(i)$.

The normalised Laplacian is defined using the adjacency matrix as follows:
\[ \Delta:=I-\D^{-{\frac {1}{2}}}\A \D^{-{\frac {1}{2}}}\; , \]
where $\D$ is the \emph{degree matrix}, i.e., a diagonal matrix given by $\D_{i i}=\deg(i)$.

\section{Permutation Entropy for Graph Signals}
\label{sec:method}
This section introduces the permutation entropy for graph signals (denoted as $\PEG$). In original $\PE$ for time series, the construction of the embedding vectors given by \Eq{eq:embPE} is made between values on consecutive steps ($t$ and $t+1$) (with the assumption $L=1$). Consecutive values cannot be defined straightforwardly in irregular graphs. 

As a motivation for the general definition, we show (with an example) how compare between values on a fixed vertex and its neighbourhoods (Section~\ref{sub:motivation}). For the general formulation, we will consider the topology of the graph encoded in the adjacency matrix to define the algorithm and construct the embedding vectors (Section~\ref{subsec:permutation-entropy}). Finally, we extend the algorithm for directed (Section~\ref{subsec:directedgraph}) and weighted graphs (Section~\ref{subsec:weightedgraph}).

\subsection{Motivation and example}\label{sub:motivation}
Consider the graph $G = (\V,\E,\A)$ and $\textbf{X}$ be any signal on the graph. Similarly to $\PE$ for time series of order $m=2$ and $L=1$, we compare the signal value at the vertex $i$ with respect the average of its neighbours, i.e., we will compare:
\[x_i \quadtext{and} \frac 1 {\deg(i)}  \sum_{j \in \Nb(i)} x_j\; .\]

Observe the relation with the normalised Laplacian, i.e.,
\[ \frac 1 {\deg(i)}\sum_{j \in \Nb(i)} x_j=x_i-\Delta x_i=(I-\Delta)x_i=\D^{-{\frac {1}{2}}}\A \D^{-{\frac {1}{2}}}x_i \;.\]

For each $i\in\V$, we define the pair where its first component is the value of the signal $\X$ on the node $i$ and the second component is the average of the signal $\X$ on the neighbours of $i$, i.e., 
\begin{equation}\label{eq:embPEG}
	\textbf{y}_i:=\left( x_i ,\left( I-\Delta\right) x_i\right) =\left( x_i, \D^{-{\frac {1}{2}}}\A \D^{-{\frac {1}{2}}}x_i\right)\;.
\end{equation}
The pair is analogous for the embedding vector defined by \Eq{eq:embPE} in $\PE$. 

We organize the $N$ pairs according to their relative values. Let $p_1$ be the number of pairs for which $x_i<\D^{-{\frac {1}{2}}}\A \D^{-{\frac {1}{2}}}x_i$, or equivalently
$\Delta x_i<0$ (represented by the
permutation $12$) and let $p_2$ be the number of the pairs for which $x_i > \D^{-{\frac {1}{2}}}\A \D^{-{\frac {1}{2}}}x_i$ or equivalently
$\Delta x_i>0$ (represented by the permutation $21$). 

We define the permutation entropy
of the graph signal $\textbf{X}$ for embedding $m=2$ and $L=1$ as a measure of the probability of the permutation $(1,2)$ and $(2,1)$, so:
\[\PEG=-\dfrac{p_1}{N}\log \dfrac{p_1}{N}-\dfrac{p_2}{N}\log \dfrac{p_2}{N}\; .\]

Intuitively, we are dividing the vertices of $G$ according to the signal $\X$ into two subsets. One set corresponds to the vertices such that $\Delta x_i>0$, i.e., it contains the \emph{local maximums} of the signal on the graph domain. Similarly, the other set contains vertices such $\Delta x_i<0$, i.e., the \emph{local minimums}. The interpretation is analogous to the permutation entropy for time series (for the case  $m=2$ and $L=1$), where the patterns are defined by the points where the function is increasing or decreasing.

\begin{example}\label{ex:permutationgraph}
	Consider the graph $G = (\V,\E,\A)$ and signal $\X$ shown in~\Fig{examplepeg}. 
	
\begin{figure}[h]
	\centering
	\subcaptionbox{\label{sub:exp1}}[.5\linewidth]
	{
		\begin{tikzpicture}[baseline,
			vertex/.style={circle,draw,fill, inner sep=0pt,minimum
				size=1mm},scale=.3]
			\node (0) at (-1,1) [vertex,label=left:$2$] {};
			\node (1) at (-1,5) [vertex,label=left:$1$] {};
			\node (2) at (2,3) [vertex,label=below:$3$] {};
			\node (4) at (4,5) [vertex,label=above:$4$] {};
			\node (3) at (4,1) [vertex,label=below:$5$] {};
			\node (5) at (6,3) [vertex,label=below:$6$] {};
			\node (6) at (9,5) [vertex,label=above:$7$] {};
			\node (7) at (9,1) [vertex,label=below:$8$] {};				
			\draw[](0) edge node[below] {} (2);
			\draw[](1) edge node[below] {} (2);
			\draw[](2) edge node[left] {} (3);
			\draw[](2) edge node[below] {} (4);
			\draw[](4) edge node[right] {} (5);
			\draw[](0) edge node[right] {} (1);
			\draw[](3) edge node[below] {} (5);
			\draw[](5) edge node[below] {} (6);
			\draw[](6) edge node[below] {} (7);
			
	\end{tikzpicture} } 
	\\
	
	\subcaptionbox{\label{sub:exp2}}[.75\linewidth]{\[\A = 
		\begin{pmatrix}
			0 & 1 & 1 & 0 & 0  & 0 & 0& 0\\
			1 & 0 & 1 & 0 & 0 & 0 & 0& 0\\
			1 & 1 & 0 & 1 & 1 & 0 & 0& 0\\
			0 & 0 & 1 & 0 & 0 & 1 & 0& 0\\
			0 & 0 & 1 & 0 & 0 & 1 & 0& 0\\
			0 & 0 & 0 & 1 & 1 & 0 & 1 & 0\\
			0 & 0 & 0 & 0 & 0 & 1 & 0& 1\\
			0 & 0 & 0 & 0 & 0 & 0 & 1 & 0\\
		\end{pmatrix}\]
	}
	\subcaptionbox{\label{sub:exp3}}[.2\linewidth]
	{\[\X =
		\begin{pmatrix}
			-1 \\
			-2.3 \\
			0 \\
			-3 \\
			1 \\
			5 \\
			1 \\
			-1.1
		\end{pmatrix}\]
	}
	\caption{An example of a graph $G$ (\ref{sub:exp1}), its adjacency matrix $\A$ (\ref{sub:exp2}) and a graph signal $\X$ (\ref{sub:exp3}). }
	\label{fig:examplepeg}
\end{figure}	
	We construct the embedding vectors given by \Eq{eq:embPEG}. We obtain one pair for each vertex, i.e., $\textbf{y}_1=\left(-1,-1.15 \right) $, $\textbf{y}_2=\left(-2.3,-0.5 \right) $, $\textbf{y}_3=\left(0.-1.325 \right) $, $\textbf{y}_4=\left(-3,2.5 \right) $, $\textbf{y}_5=\left(1,2.5 \right) $, $\textbf{y}_6=\left(5,-0.333 \right), \textbf{y}_7=\left( 1,1.95\right)$ and
	$\textbf{y}_8=\left( -1.1,1\right)$. 
	
	We have two patterns for the case $m=2$. The pairs $\textbf{y}_2, \textbf{y}_4, \textbf{y}_5, \textbf{y}_7$ and  $\textbf{y}_8$  belong 
	to the same pattern (where $x_i < \D^{-{\frac {1}{2}}}\A \D^{-{\frac {1}{2}}}x_i$) and $\textbf{y}_1, \textbf{y}_3$ and $\textbf{y}_6$ belong 
	to the second pattern (where $x_i > \D^{-{\frac {1}{2}}}\A \D^{-{\frac {1}{2}}}x_i$). 
	
	The relative frequency of each permutation pattern is $\frac{5}{8}$ and $\frac{3}{8}$ respectively. Finally using Shannon's entropy,
	the $\PEG$ value of the signal $\X$ is equal to $-\frac{5}{8}\ln\left( -\frac{5}{8}\right) -\frac{3}{8}\ln\left( -\frac{3}{8}\right)=0.6616$. The normalised $\PEG$ is $\frac{0.6616}{\ln(2)}=0.9544$. 
\end{example}

\subsection{Permutation entropy for graphs signals}\label{subsec:permutation-entropy}

Let $G = (\V,\E,\A)$ be a graph and $\textbf{X}=\left\{x_i\right\}_{i=1}^{N}$ be a signal on the graph, the permutation entropy for the graph signals $\PEG$ is defined as follow: 

\begin{enumerate}
	\item For $2\leq m\in\N$ the \emph{embedding dimension} and $L\in\N$ the \emph{delay time}, we construct the embedding vector $\textbf{y}_i^{m,L}\in\R^m$ given by
	\begin{equation*}
		\textbf{y}_i^{m,L}=\left( y_{i}^{kL}\right)_{k=0}^{m-1}=\left(y_i^0,y_{i}^{L},\dots y_{i}^{(m-1)L}\right)\; ,
	\end{equation*}
	for all $i=1,2,\dots,N$ and where

\begin{align}
	y_{i}^{kL}&= \frac 1 {\card{\Nb_{kL}(i)}}  \sum_{j \in \Nb_{kL}(i)} x_j\label{eq:average}\\
	          & =\frac{1}{\card{\Nb_{kL}(i)}}(\A^{kL}\X)_i\;.\label{eq:averagem}
\end{align}
	Recalling $\Nb_{kL}$ is defined by \Eq{eq:nb}, it follows that $y_i^0=x_i$ and $y_i^1=(I-\Delta)x_i$.
	\item The $m$ real numbers of the embedding vector $\textbf{y}_i^{m,L}$ are associated with integer numbers from $1$ to $m$ and then arranged in an increasing order.
	
	There are $m!$ permutation (also called possible patterns) $\pi$ for an $m$-embedding vector.
	\item  The relative frequency for the distinct permutation $\pi_1,\pi_2,\dots,\pi_k$ where $k\leq m!$ is denoted by $p(\pi_1),p(\pi_2),\dots,p(\pi_k)$. The permutation entropy $\PEG$ for the time series $\textbf{X}$ is computed as the Shannon entropy for the $k$ distinct permutations
	$$
	\PEG (m, L)=-\sum_{i=1}^{k} p(\pi_i) \ln p(\pi_i)\;.
	$$
	
\end{enumerate}

In the next sections, without specification $L=1$ is chosen. If all possible dispersion pattern have equal probability value, the $\PEG$ reaches its highest value which is equal to $\ln(m!)$. Note that we use the normalised $\PEG$ as $\dfrac{\PEG}{\ln(m!)}$.

We use \Eq{eq:average} to prove some properties of $\PEG$, while \Eq{eq:averagem} is more useful for a numerical implementation. Along this article, we will also assume that $G$ has no isolated vertices to avoid $\card{\Nb_{kL}(i)}=0$ in both equations.

	Some of the similarities and differences between the original $\PE$ and the permutation entropy for graph signals $\PEG$ are the following:
	\begin{enumerate}
		\item  The main difference is the construction of the embedding vectors in the step $1$. With \Eq{eq:embPE}, $\PE$ constructs $N-(m-1)L$ embedding vectors using $m$ consecutive values of the signal.  $\PEG$
		uses the adjacency matrix and \Eq{eq:embPEG} to obtain $N$ embedding vectors (independent of $m$ and $L$), each embedding vector corresponds to one vertex. 	
		\item The step $2$ (arrange the embedding vectors in increasing order) and $3$ (computing the Shannon's entropy) for both algorithms $\PE$ and $\PEG$ are exactly the same. Therefore, the computational cost is almost the same for both algorithms.
		\item 
		A time series $\X$ can be considered as a graph signal over the graph $G$ given by the undirected path. The value given by $\PE$ for the time series $\X$ and the value given by the $\PEG$ when $\X$ is considered as a graph signal are (in general) different. However, the dynamics detected by the $\PE$ is preserved with $\PEG$, see \Sec{exp}.
		
		The \emph{arrow of time} information is lost when we consider the undirected path. One way to preserve the information is consider $\X$ as a graph signal on the directed path and define $\PEG$ for directed graphs (see Section~\ref{subsec:directedgraph}).
		\item The adjacency matrix $\A$ is a particular case of the weight matrix $\W$, where all non-zero weight values are equal to one. The general algorithm for weighted graphs is presented in Section~\ref{subsec:weightedgraph}.
	\end{enumerate}

The entropy of the graph signal $\textbf{X}$ depends on its numerical values and the graph topology. It is interesting to study these quantities (that depend on the geometric structure of the graph) under some geometric perturbation (delete edges, vertices and contractions \cite{Fabila-Carrasco2020}). We will show that under some conditions, adding or deleting edges on the graph will preserve the permutation entropy of the signal $\X$.

\begin{proposition}\label{prp:adding}
	Let  $\textbf{X}$ be a graph signal over the graph $G=(\V,\E,\A)$ with entropy value $\PEG$ for $m=2$ and $L=1$. Let $i,j\in \V$ be two vertices such that $\Delta x_i<0<\Delta x_j$.  
	\begin{enumerate}
		\item If $\{i,j\}\notin\E$ and  $x_i<x_j$, then $PE_G=PE_{G'}$ where $G'=G+\{i,j\}$.
		\item If $\{i,j\}\in\E$ and  $x_j<x_i$, then $PE_G=PE_{G'}$ where $G'=G-\{i,j\}$.
	\end{enumerate}
\end{proposition}
\begin{proof}
	We will prove $1$ and the proof for $2$ is similar.
	\begin{enumerate}
		\item First, we will prove that $\Delta x_i<0$ together with $x_i<x_j$ implies $\Delta_{G'} x_i<0$ where $G'=G+\{i,j\}$. It follows by:
		\begin{align*}
			\Delta x_i&<0 \\
			deg_G(i) x_i&< \sum_{k \in \Nb_G(i)}x_k\\
			x_i \left( deg_G(i)+1\right)&< \sum_{k \in \Nb_G(i)}x_k+x_j \\
			x_i&<\dfrac{1}{deg_{G'}(i)}\sum_{k \in \Nb_{G'}(i)}x_k\\
			\Delta_{G'} x_i&<0\;.
		\end{align*}
	In an analogous way, it can be shown that $0<\Delta x_j$ and $x_i<x_j$ imply that $0<\Delta_{G'} x_j$. Therefore, for each vertex $k\in \V$ the pairs $\left(x_k, (I-\Delta)x_k \right) $ preserve the same order on both graphs ($G$ and $G+\{i,j\}$). The relatives frequencies are equals, hence, their entropy are equal.	
	\end{enumerate}	
\end{proof}
In the previous result, we prove a condition that preserves not only the entropy value but their relative frequencies. Hence, we can apply in an iterative way to generalise the result with the following corollary.
\begin{corollary}\label{cor:edges}
	Let  $\textbf{X}$ be a graph signal over the graph $G=(\V,\E,\A)$ with entropy value $\PEG$ for $m=2$ and $L=1$. Define the sets with the following property
	\begin{align}
		E_0=\set{(i,j)}{i\in A, j\in B \text{ and } x_i<x_j}\label{eq:condition1}\\
		E_1=\set{(i,j)}{i\in A, j\in B \text{ and } x_j<x_i}\label{eq:condition2}
	\end{align}
	where $A=\set{i\in \V}{\Delta x_i<0}$
	and $B=\set{i\in \V}{0<\Delta x_i}$. 
	\begin{itemize}
		\item If $E'\subset E_0$ and $E'\cap \E=\emptyset$, then $PE_G=PE_{G'}$ where $G'=G+E'$.
		\item If $E'\subset E_1$ and $E'\subset \E$, then $PE_G=PE_{G'}$ where $G'=G-E'$.
	\end{itemize}
	
\end{corollary}

\begin{example}
	Consider the graph $\G$ and signal $\X$ given in \Ex{permutationgraph}. Define the sets $A=\set{i\in \V}{\Delta x_i<0}=\{2,4,5,7,8\}$ and $B=\set{i\in \V}{0<\Delta x_i}=\{1,3,6\}$.
	
	The edge set given by $E_0=\{(1,4),(1,8),(3,8),(6,2),(6,8)\}$ fulfil the condition in \Eq{eq:condition1}. Define $G'=G+E'$ for any $E'\subset E_0$, then by \Cor{edges} follows the entropy is the same, i.e., $PE_G=PE_{G'}$. The case $E'=E_0$ is shown in Fig.~\ref{sub:G1}. 
	
	Similarly, $E_1=\{(3,5)\}$ (fulfil the condition in \Eq{eq:condition2}), then $G$ and $G-E_1$ have the same entropy. The graph $G-E_1$ is shown in Fig.~\ref{sub:G2}.

	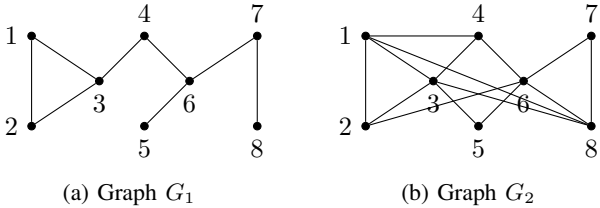
\begin{figure}[h]
		\subcaptionbox{Graph $\G_1$\label{sub:G1}}[.4\linewidth]		{\begin{tikzpicture}[baseline,
				vertex/.style={circle,draw,fill, inner sep=0pt,minimum
					size=1mm},scale=.3]
				\node (0) at (-1,1) [vertex,label=left:$2$] {};
				\node (1) at (-1,5) [vertex,label=left:$1$] {};
				\node (2) at (2,3) [vertex,label=below:$3$] {};
				\node (4) at (4,5) [vertex,label=above:$4$] {};
				\node (3) at (4,1) [vertex,label=below:$5$] {};
				\node (5) at (6,3) [vertex,label=below:$6$] {};
				\node (6) at (9,5) [vertex,label=above:$7$] {};
				\node (7) at (9,1) [vertex,label=below:$8$] {};				
				\draw[](0) edge node[below] {} (2);
				\draw[](1) edge node[below] {} (2);
				\draw[](2) edge node[below] {} (4);
				\draw[](4) edge node[right] {} (5);
				\draw[](0) edge node[right] {} (1);
				\draw[](3) edge node[below] {} (5);
				\draw[](5) edge node[below] {} (6);
				\draw[](6) edge node[below] {} (7);
				
		\end{tikzpicture} } 
		\subcaptionbox{Graph $\G_2$\label{sub:G2}}[.6\linewidth]
		{\begin{tikzpicture}[baseline,
				vertex/.style={circle,draw,fill, inner sep=0pt,minimum
					size=1mm},scale=.3]
				\node (0) at (-1,1) [vertex,label=left:$2$] {};
				\node (1) at (-1,5) [vertex,label=left:$1$] {};
				\node (2) at (2,3) [vertex,label=below:$3$] {};
				\node (4) at (4,5) [vertex,label=above:$4$] {};
				\node (3) at (4,1) [vertex,label=below:$5$] {};
				\node (5) at (6,3) [vertex,label=below:$6$] {};
				\node (6) at (9,5) [vertex,label=above:$7$] {};
				\node (7) at (9,1) [vertex,label=below:$8$] {};				
				\draw[](0) edge node[below] {} (2);
				\draw[](1) edge node[below] {} (2);
				\draw[](2) edge node[left] {} (3);
				\draw[](2) edge node[below] {} (4);
				\draw[](4) edge node[right] {} (5);
				\draw[](0) edge node[right] {} (1);
				\draw[](3) edge node[below] {} (5);
				\draw[](5) edge node[below] {} (6);
				\draw[](6) edge node[below] {} (7);
				
				\draw[](1) edge node[right] {} (4);
				\draw[](1) edge node[right] {} (7);
				\draw[](2) edge node[below] {} (7);
				\draw[](5) edge node[below] {} (0);
				\draw[](5) edge node[below] {} (7);				
				
		\end{tikzpicture} }

		\caption{ Examples of two graphs: $\G_1$ and $\G_2$. $\G_1$ is a subgraph of $\G$ and $\G$ is a subgraph of $\G_2$. Both graphs $\G_1$ and $\G_2$ preserve the entropy of the original graph $G$. }
		\label{fig:examplepeg2}
	\end{figure}
\end{example}
	In this sense, given a signal over a graph, with \Prp{adding} we can find structures that preserve not only the same numeric entropy value but the \textit{maximal} and \textit{minimal} values of the signal on the same vertices.  

 The \textit{invariance property} with respect to monotonic transformation of the time signal is an important property of the $\PE$, i.e., if $\X$ is a time series, and $f$ is an arbitrary strictly decreasing (or increasing) real function, then the classical $\PE$ of the time series $\X$ and $f(\X)$ are equal \cite{Bandt2002}. This function occurs, for example, when the data is measured with different equipment. In a similar scenario, the following proposition shows that some modification on the signal does not change the permutation entropy.   

\begin{proposition}\label{prp:cX}
	Let  $\textbf{X}$ be a graph signal over the graph $G$ and $c$ a real (non zero) constant function defined on the vertices. The entropy of the signals: $\X$, $c\X$ and $c+\X$ are equal.
\end{proposition}
\begin{proof}
	For any $2\leq m\in\N$ and $L\in\N$, the embedding vector for the graph signal $\X$ are defined as $\textbf{y}_i^{m,L}=\left( y_{i}^{kL}\right)$ (see \Eq{eq:average}). It is easy to show that the embedding vectors for $c\X$ are $c\textbf{y}_i^{m,L}$ and the embedding vectors for $c+\X$ are $c+\textbf{y}_i^{m,L}$. Therefore, the proportion of the patterns in the original signal $\X$ are preserved in the signals $c\X$ and $c+\X$. Therefore, its entropy values are equal. 
\end{proof}

The previous proposition shows a difference with respect to the definition of \textit{smoothness} on graph signals. Formally, for a graph signal $\X$ on the $\G$, the smoothness is measured in terms of the quadratic form of the normalised Laplacian
\begin{equation}\label{eq:smooth}
	\X^T\Delta\X:=\dfrac{1}{2}\sum_{i\sim j}(x_i-x_j)^2\;.
\end{equation}

Therefore, the smoothness of the signal $c\X$ is different (in general) from the smoothness of the signal $\X$. The algorithm $\PEG$ is interested in the change of the patterns rather than the changes of values of the signal as in the smoothness definition. 
%

\subsection{Permutation entropy for signals on directed graph}\label{subsec:directedgraph}
In the Section~\ref{subsec:permutation-entropy} we introduced the permutation entropy for undirected graphs. As a particular case, in this section, we introduce a permutation entropy algorithm for directed graphs, denoted as $\PE_{\DG}$. 

A \emph{directed graph} or \emph{digraph} is a graph where each edge has an orientation or direction. The directed edge (called also an \emph{arc}) is an order pair $(i,j)$ and it is drawn as an arrow from the vertex $i$ to the vertex $j$.
A \emph{directed path} on $k$ vertices is a directed graph which joins a sequence of different vertices with all the edges in the same direction and is denoted as $\overrightarrow{P}$, i.e. its vertices are $\{1,2,\dots,k\}$ and its arcs $(i,i+1)$ for all $1\leq i \leq k-1$.

 The permutation entropy for signals on directed graphs will be almost identical to the presented in Section~\ref{subsec:permutation-entropy} except for a small change in the construction of the embedding vector.
\begin{enumerate}
	\item For $2\leq m\in\N$ the \emph{embedding dimension} and $L\in\N$ the \emph{delay time}, define the set
\begin{equation}\label{eq:directedset}
	V^*=\set{i\in \V }{ \overrightarrow{\Nb}_{(m-1)L}(i)\neq \emptyset }\:,
\end{equation}	
	where \begin{equation}\label{eq:set}
		\overrightarrow{\Nb}_k(i)=\set{j\in\V}{\scalebox{.8}[1.0]{it exists a directed path on $k$ edges from $i$ to $j$}}.
	\end{equation}
	 We construct the embedding vector $\textbf{y}_i^{m,L}\in\R^m$ given by
\begin{equation}\label{eq:embdir}
	\textbf{y}_i^{m,L}=\left( y_{i}^{kL}\right)_{k=0}^{m-1}=\left(y_i^0,y_{i}^{L},\dots y_{i}^{(m-1)L}\right)\; ,
\end{equation}
for all $i\in V^*$  where
\begin{align}
	y_{i}^{kL}&= \frac 1 {\card{\overrightarrow{\Nb}_{kL}(i)}}  \sum_{j \in 	\overrightarrow{\Nb}_{kL}(i)} x_j\;. \label{eq:average2}
\end{align}
\end{enumerate}

The step $2$ and $3$ are the same as in Section~\ref{subsec:permutation-entropy}. The next proposition shows that the classical permutation entropy is the same if we consider the time series as a signal over a directed path. Therefore, we generalise the $\PE$ for all directed graphs with the same values that the original one. 
\begin{proposition}\label{prp:equal}
	Let $\textbf{X}=\left\{x_i\right\}_{i=1}^{N}$ be a time series and consider $G=\overrightarrow{P}$ the directed path on $N$ vertices, then 	for all $m$ and $L$ the equality holds:
	\[\PE(m,L)=\PE_{G}(m,L)\;.\]
\end{proposition}
\begin{proof}
For the embedding dimension $m$, delay time $L$ and $\G$ the directed path with $N$ vertices, then $\overrightarrow{\Nb}_k(i)=\{i+k\}$ for all $1\leq k\leq N-i$ and $\emptyset$ otherwise (see~\Eq{eq:set}). 

The set defined in \Eq{eq:directedset} is $V^*=\{1,2,3,\dots,N-(m-1)L \}$. Then, for all $i\in V^*$, by \Eq{eq:average2}:
\begin{align*}
	y_{i}^{kL}&= \frac 1 {\card{	\overrightarrow{\Nb}_{kL}(i)}}  \sum_{j \in	\overrightarrow{\Nb}_{kL}(i)} x_j=x_{i+kL}\;. 
\end{align*}
Therefore, the embedding vector $\textbf{y}_i^{m,L}\in\R^m$ defined by \Eq{eq:embdir} is 
\begin{equation*}
	\textbf{y}_i^{m,L}=\left( y_{i}^{kL}\right)_{k=0}^{m-1}=\left(x_i,x_{i+L},\dots x_{i+(m-1)L}\right)\; ,
\end{equation*}
for all $i\in V^*=\{1,2,3,\dots,N-(m-1)L \}$, hence the are exactly the same embedding vectors defined in the original $\PE$ in \Eq{eq:embPE}, i.e. $\textbf{y}_i^{m,L}=\textbf{x}_i^m(L)$. Because the step 2) and 3) in both algorithms are the same, we conclude $\PE=\PEG$.
\end{proof}

\subsection{Permutation entropy for signals on weighted graph.}\label{subsec:weightedgraph}
In most of the examples, the adjacency matrix usually it is enough. Nevertheless, the previous results and algorithms can be generalised for weighted graphs.

	A \emph{weighted undirected graph} $G$ is defined as the triple $G = (\V,\E,\W)$
	which consist of a finite set of vertices or nodes $\V=\{1,2,3,\dots, N\}$, an edge set $\E = \{(i,j): i,j\in\V\}$  and $\W$
	is the corresponding $n\times n$ symmetric adjacency matrix weighted on edges with entries $0\leq w_{i j}=w_{j i}$ the weight of edge $(i,j)$.

Observe that $(\W^k) _{ij}$ is the sum of the product of all the weights of all the walks from the vertex $i$ to the vertex $j$ of length exactly $k$. We define a function on the vertices $\map{\deg^k}{\V}{\R}$ given by 
\begin{equation}\label{eq:weightver}
	\deg^k(i):=\sum_{j\in \V}(\W^k) _{ij}=\sum_{j\in \V} (\W^k)_{ji}\; .
\end{equation}

Let $\textbf{X}=\left\{x_i\right\}_{i=1}^{N}$ be a signal on the graph $G = (\V,\E,\W)$, the permutation entropy of the signal $\textbf{X}$ on the weighted graph $\G$ is the same that the presented in Section~\ref{subsec:permutation-entropy}, but replace the \Eq{eq:average} by the following:

\begin{equation*}
	y_{i}^{kL}	=\frac{1}{\deg^k(i)}(\W^{kL}\X)_i\;.
\end{equation*}
Similarly to Section~\ref{subsec:directedgraph}, we can extend the algorithm for weighted directed graphs.
\section{Experiments and discussion}
\label{sec:exp}
In this section, we apply our algorithm to synthetic and real data, including signals on 1D (time series), 2D (image) and irregular domains (graph). We show that $\PEG$ is a suitable generalisation of the original $\PE$ for time series, but with the advantage that the input could be any graph signal. 

In 1D, the equality $\PEG=\PE$ holds if the underlying graph $G$ is the directed path (\Prp{equal}). The values differ slightly when $G$ is an undirected graph. However, $\PEG$ can detect different dynamics for synthetic signals (logistic map, autoregressive models) and for real signals (fantasy data set). In case the input is 2D, our algorithm gives similar results to the recently introduced two-dimensional permutation entropy~\cite{Morel2021}.  

Finally, we apply the algorithm for signals defined in general graphs. Fixing the underlying graph $\G$ and consider a signal $\X$, we show the $\PEG$ value depends on the irregularity of the signal. We also consider the reverse case, fixing the signal $\X$ (for example, a set of $n$ random values). We consider several underlying graphs $\G$ with $\X$ as a graph signal, we study the impact on the graph structure has on the entropy measure.  

\subsection{Examples on 1D and the classical permutation entropy}\label{subsec:1D}
Let $\textbf{X}=\left\{x_i\right\}_{i=1}^{N}$ be a time series. In this section we consider three underlying graphs: $\G_1$ a directed path, $\G_2$ an undirected path, and $\G_3$ a directed path with the reverse orientation respect $\G_1$. For any $m$ and $L$, recall that the classical $\PE$ gives equal results as $\PE_{G_1}$ (\Prp{equal}), hence the permutation for graph signals has the same properties as the classic one in 1D. Moreover, we verified computationally that the values of $\PE_{G_1}$ are almost exactly the same as $\PE_{G_3}$, and for $N>30$, its difference is always in the order of $10^{-16}$ (computational accuracy of Matlab).
\subsubsection{The logistic map}\label{subsubsec:logisticmap}

$\PE$ has been used to detect dynamical changes in time series \cite{Bandt2002, Cao2004}. A commonly used example to show the performance of $\PE$ is the logistic map, given by
\[x_{n+1}=r x_n(1-x_n)\;.\]

The analysis is relevant for the parameter $r$. Thus, we vary the parameter $r$ from $3.55$ to $4.0$ with increments $r$ in steps of $10^{-5}$ at each iteration, we define the sequence given by $\X(r)=\{x_i\}_{i=1}^N$. The initial value is $x_0=0.65$ and we consider $N=2^{14}$ points. Fig.~\ref{sub:logisticA} shows the time series, where each point of the discrete time is plotted for each value of $r$, i.e., the bifurcation diagram for the logistic map for  $r\in[3.55, 4.0]$.

We created $4501$ time series, each consisting with $2^{14}$ points. For each sequence, we consider two underlying graph: $\G_1$ a directed path and $G_2$ an undirected path, both on $N$ vertices. Finally, we compute its permutation entropy for $m=3$ and $L=1$ (see Fig.~\ref{sub:logisticB}). 

\begin{figure}
	\begin{subfigure}[h]{0.5\textwidth}
		\hspace*{-.425cm}\includegraphics[scale=.29]{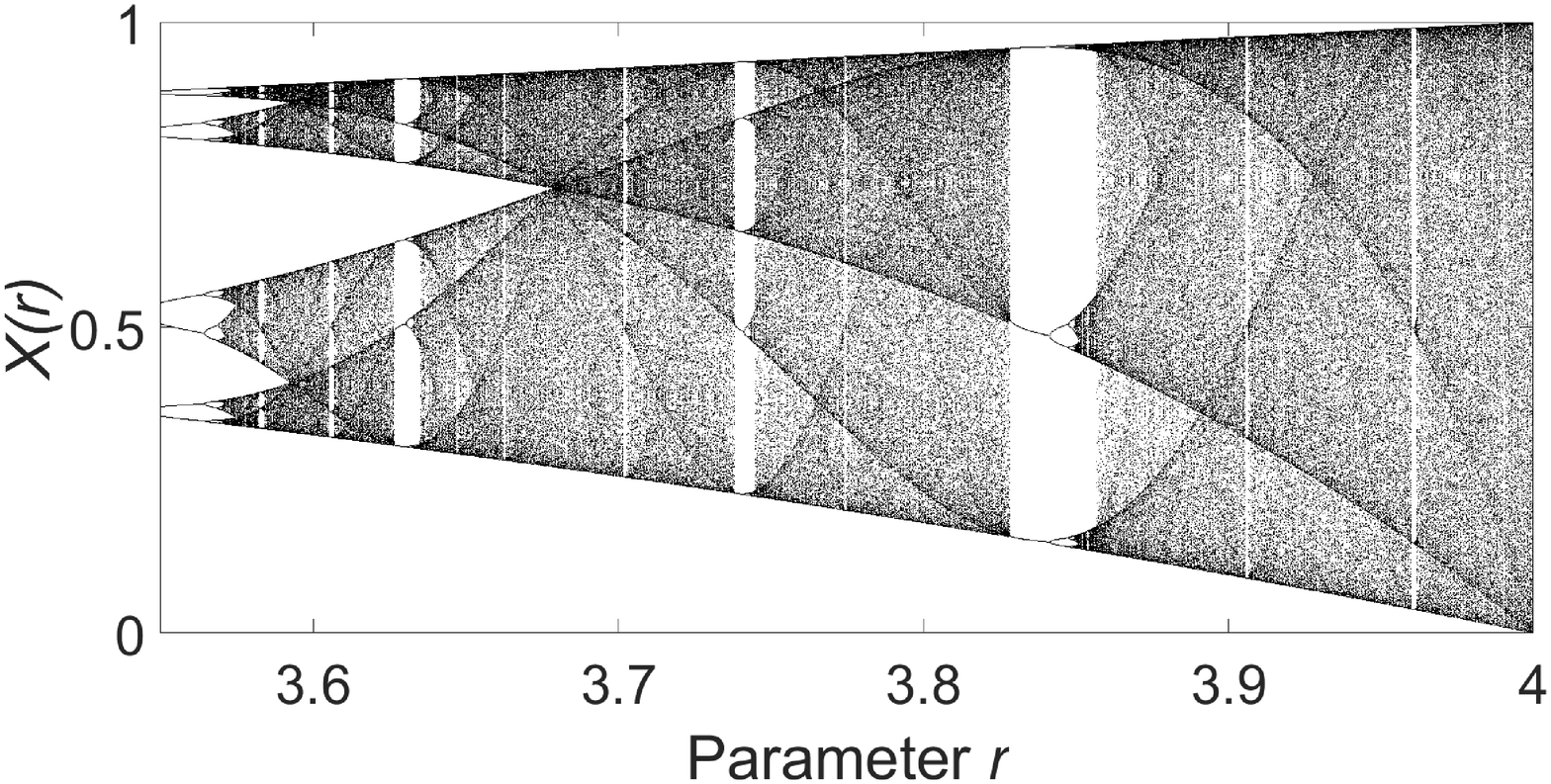}
		\caption{The logistic map with parameter $r$ changing from $3.5$ to $4$.}
		\label{sub:logisticA}
	\end{subfigure}\hfill
	\begin{subfigure}[b]{0.5\textwidth}
		
		\includegraphics[scale=.29]{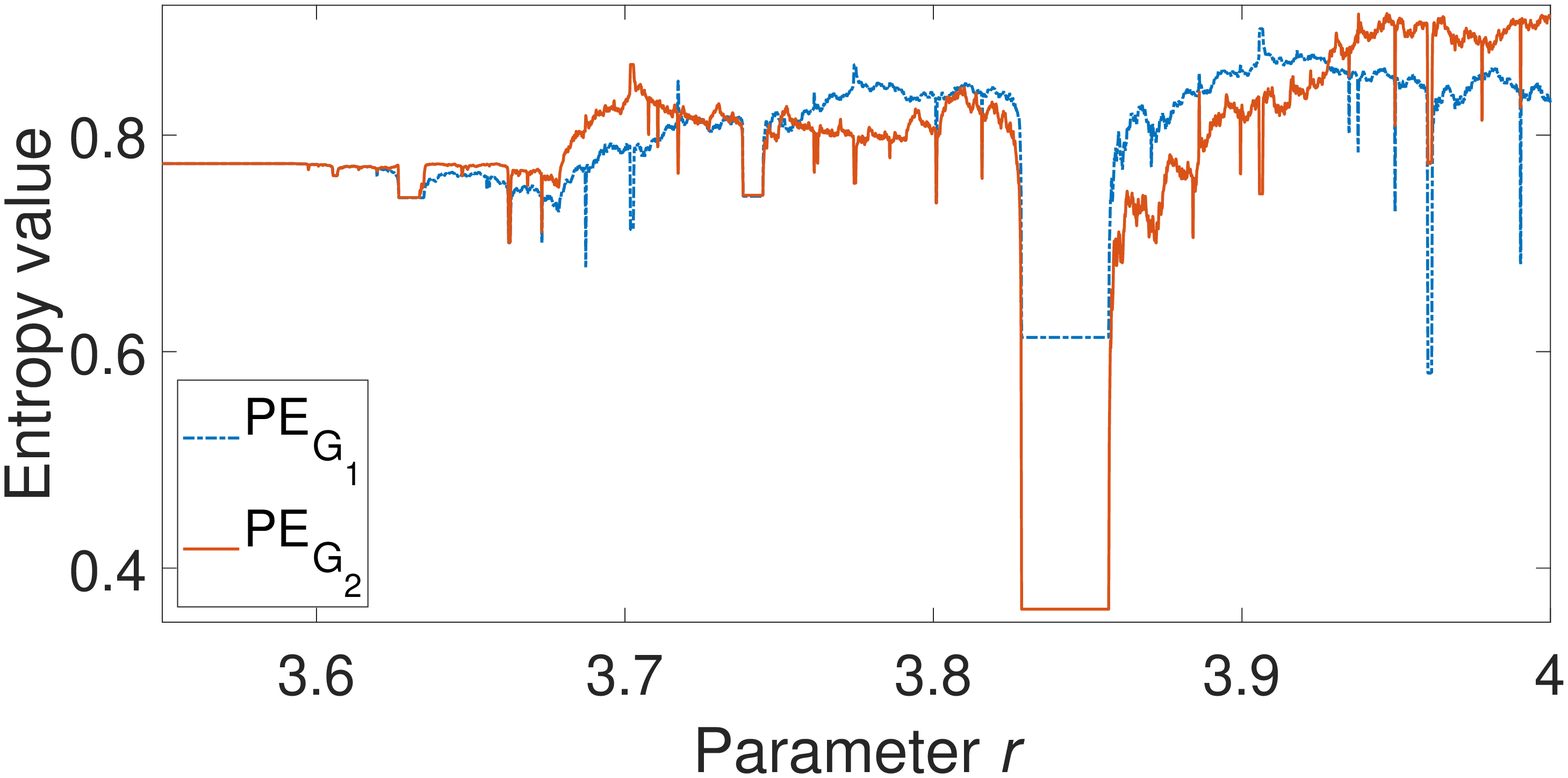}
		\caption{Entropy values computed with underlying graph $\G_1$ (directed path) and $\G_2$ (undirected path). Both algorithms use $m=3$ and $L=1$.}
		\label{sub:logisticB}	
	\end{subfigure}
	\caption{ }
	\label{fig:logistic}
\end{figure}
It is known that chaotic behaviour starts for $3.5699\leq r \leq 4$. The entropy algorithm is able to detect \emph{island of stability}, i.e., values of $r$ such the data sequence shows non-chaotic behaviour. The largest window is $1+\sqrt{8}\approx3.8284<r<3.8415$, this range of $r$ shows oscillation among three values \cite{Zhang2010}. The algorithm (with both underlying graphs $\G_1$ and $\G_2$) detects the window (for any embedding dimension $m$). However, the wider gap between the values for $\G_2$ indicates a large sensitivity of the algorithm to detects the changes of the dynamic on the data. A similar effect in other islands of stability occurs. This fact is in agreement with other previous studies \cite{Bandt2002, Cao2004}.   
  	
\subsubsection{Heart beat time series} The Fantasia database has been analysed widely to validate the performance of some entropies algorithms~\cite{Chen2019,Bian2012}. We use $10$ heart beat time series: $5$ correspond to young subjects (aged between 21 and 34 years) and 5 recordings from elderly subjects (aged between 68 and 85 years). Each time series is divided into samples of $800$ points with an overlap of $200$ points. The classical $\PE$ is computed (or equivalently, its permutation graph entropy for the directed path) for each sample. We also consider each time series as a graph signal on the undirected path, and $\PEG$ is computed for each case. We consider the embedding dimensions $2\leq m\leq 8$ for the computation. The averaged entropy values with their standard error bars are shown in \Fig{fheart}. 

The analysis shows that the elderly and young subjects are not indistinguishable by the classical $\PE$ for $3\leq m \leq 6$. Considering $\X$ as a graph signal on the undirected path, the algorithm $\PEG$ can difference the subjects for all embedding dimensions (except $m=2$). Changing the size of the samples and/or intersection does not change this behaviour. In addition, we observe that the entropy values of the elderly subjects are consistently higher than the young subjects for all embedding dimensions with $\PEG$. In contrast, the order in $\PE$ values depends on the parameter $m$, that is, the ranking of elderly and young people is not consistent.

\begin{center}
	\begin{figure}
	\centering
	\includegraphics[scale=.29]{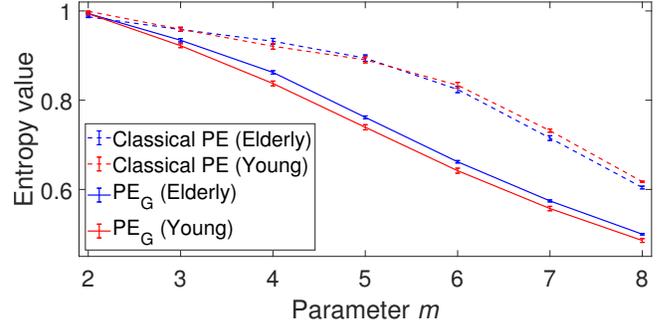}
	\caption{Averaged entropy values and standard error bars for embedding dimensions $2\leq m\leq 8$.}
	\label{fig:fheart}	
\end{figure}
\end{center}

\subsection{Permutation entropy for images (2D)}\label{subsec:1D}
One of the main advantages of our algorithm is the fact that it can be applied on any graph, including the directed graph shown in Fig. \ref{sub:directed-graph} (or its undirected version), where any signal can be regarded as an image. Therefore the permutation entropy (described in Section~\ref{subsec:directedgraph}) gives us a metric of the regularity/complexity of images. 

For the analysis, we choose the directed graph $\overrightarrow{G}$ because: 1) the directed adjacency matrix has more entries equal to zero than the undirected version and hence the algorithm is faster, 2) the algorithm $\PE_{2D}$ presented in~\cite{Morel2021} (and almost every 2D algorithm) implicitly uses this orientation in the vectorisation, 3) the orientation preserves more information of the \textit{geometry} of the graphs and hence gives us better results, 4) if $\overrightarrow{G}$ is a 2D graph with size $N\times1$ or $1\times N$, then $\PE_{\overrightarrow{G}}$ is equal to the classical $\PE$, hence, our algorithm is a natural generalisation and 5) choosing vertex $1$ or any other vertex as an origin of the orientation gives almost identical results, because of its symmetry.  

To assess the ability of $\PEG$ and similarly to \cite{Silva2016,Azami2017a}, we use the two-dimensional processing $\MIX_{2D}$.
\subsubsection{$\MIX_{2D}$  processing}
Let $X_{i,j}=\sin\left( \frac{2 \pi i}{12}\right)+\sin\left( \frac{2 \pi j}{12}\right)$ and let $Z_{i,j}$ be a random variable where $Z_{i,j}=0$ with probability $1-p$ and $Z_{i,j}=1$ with probability $p$. In addition, consider $Y_{i,j}$ a matrix of random values ranged in $[-\sqrt 3,\sqrt 3]$. The $\MIX_{2D}$ process is defined by the equation:
\begin{equation}\label{eq:mix}
	MIX_{2D}(p)_{i,j}=\left( 1-Z_{i,j}\right) X_{i,j}+Z_{i,j} Y_{i,j}\;. 
\end{equation}
\Fig{mixwithp} shows samples for different values of $p$, and the underlying graph consider.

\begin{figure}
	\centering
	\subfloat[Illustration of digraph $\overrightarrow{G}$]{\begin{tikzpicture}[scale=.8,baseline=1mm]
			\node[label={\small 1}] (a) at (0,2) {};
			\fill[color=black] (0,2) circle (0.04);
			\foreach \x in {0,1,...,2} {
				\foreach \y in {0,1,...,2} {
					\fill[color=black] (\x,\y) circle (0.04);
				}
			}
			\foreach \x in {0,1,...,2} {	
				\foreach \y in {0,1} {
					
					\path [>=latex,<-] (\x,\y) edge (\x,\y+1);
				}
			}
			\foreach \x in {0,1} {	
				\foreach \y in {0,1,2} {

					\path [>=latex,->] (\x,\y) edge (\x+1,\y);
				}
			}
			\foreach \x in {0,1} {
				\foreach \y in {1,2} {	
					\path [>=latex,<-] (\x+1,\y-1) edge (\x,\y);
					\path [>=latex,->] (\x+1,\y) edge (\x,\y-1);
				}
			}
		\end{tikzpicture}
		\label{sub:directed-graph}}\quad\quad
	\subfloat[$p=0$]{\includegraphics[scale=.5]{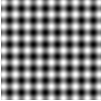}%
		\label{fig_first_case}}\quad\quad
	\subfloat[$p=0.1$]{\includegraphics[scale=.5]{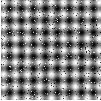}%
		\label{fig_second_case}}\quad\quad\\
	\subfloat[$p=0.25$]{\includegraphics[scale=.5]{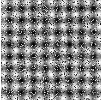}%
		\label{fig_first_case}}\quad\quad
	\subfloat[$p=0.5$]{\includegraphics[scale=.5]{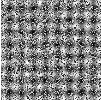}%
		\label{fig_first_case}}\quad\quad
	\subfloat[$p=0.9$]{\includegraphics[scale=.5]{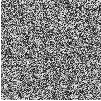}%
		\label{fig_first_case}}
	\caption{Examples of images generated by the MIX process in \Eq{eq:mix}. The size of each image is $100\times 100$ pixels.}
	\label{fig:mixwithp}
\end{figure}
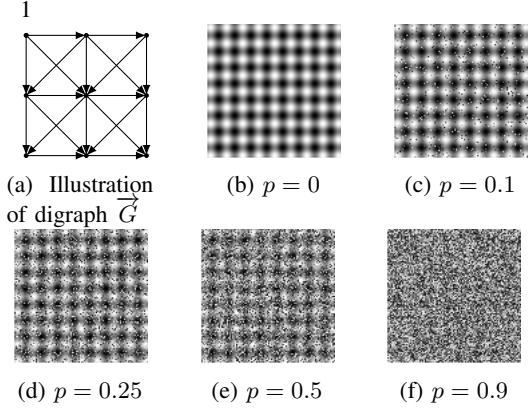
To understand the effect of size of image, we create 10 different realization of $\MIX_{2D}(p)$ for each $p=0.1,0.25,0.5,0.9$ whose size changed from $10\times 10$ to $100\times 100$ (for larger size, the results are similar). For each realisation, we compute its $\PE_{\DG}$ (Fig.~\ref{sub:directed-graph}) with embedding dimension $m=6$. In \Fig{allp} is shown the mean and standard deviation values of $\PE_{\DG}$. We also compute the $\PE_{2D}$ with embedding dimension $d_x=3$ and $d_y=2$  (see~\cite{Morel2021}). In both methods, $6!$ permutation patterns are possible and the results are similar.
\begin{figure}[h]
	\centering
	\includegraphics[scale=.29]{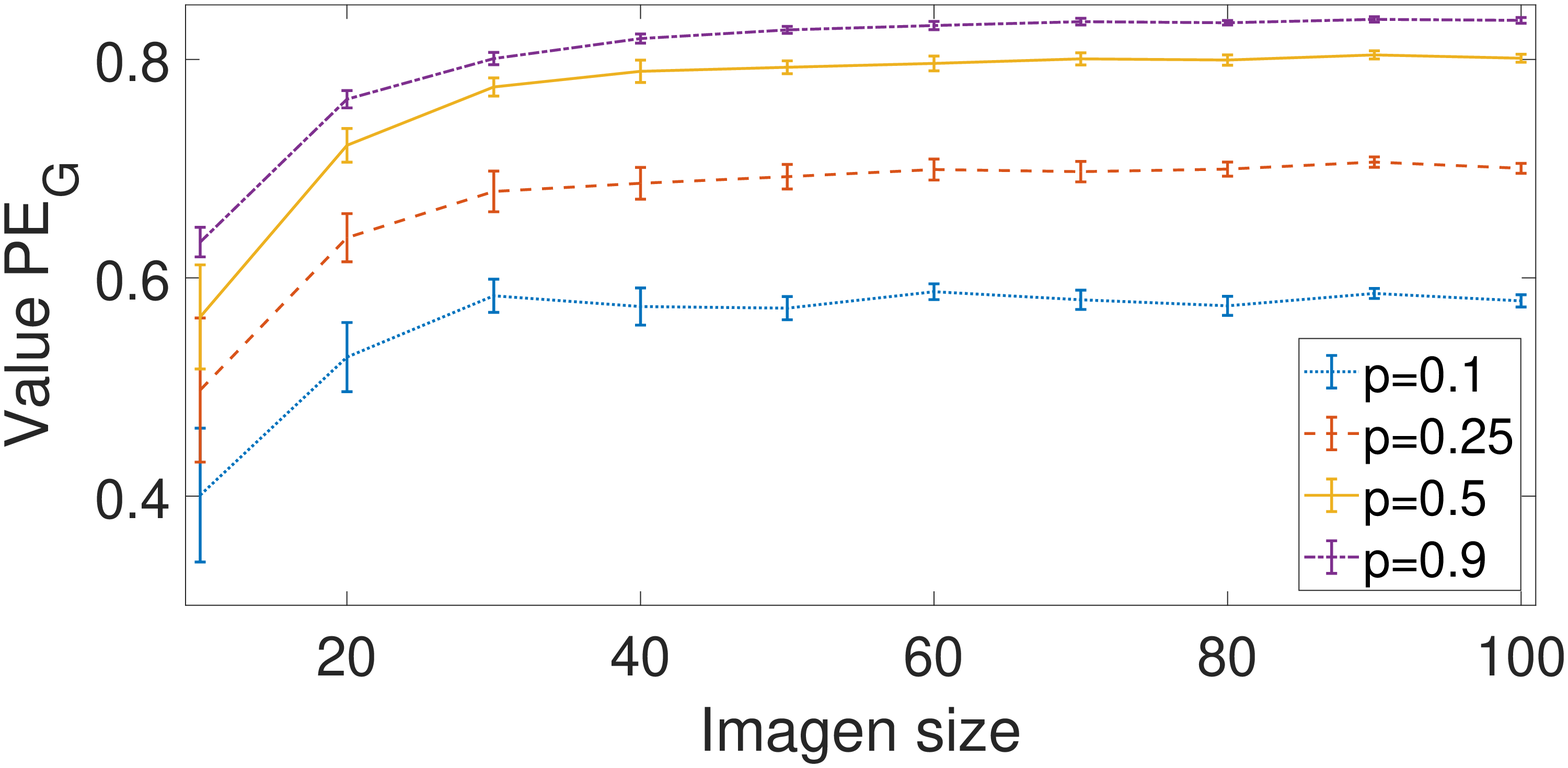} 
	\caption{Mean value and standard deviation of results obtained by $\PEG$ computed from $20$ realisations from MIX processing.}
	\label{fig:allp}
\end{figure}

As $p$ goes to $1$, the image gets closer to random noise (more irregular) and then the $\PE_{\DG}$ gets close to $1$. In particular, if $r<s$, then the entropy value of $\MIX_{2D}(r)$ is smaller than $\MIX_{2D}(s)$. These results are in agreement with the bidimensional entropy algorithms as the distribution \cite{Azami2017a} and sample entropy \cite{Silva2016}. 

\subsubsection{Artificial periodic and synthesized textures}
We use the same six periodic textures and their corresponding synthesized textures (as in~\cite{Morel2021}) to show how $\PEG$ changes when a periodic turns into a synthesized texture. The images dataset are downloaded from \cite{web1}.  The original textures and their corresponding texture (sized $256\times 256$) are depicted in the same order in \Fig{texture}.

\begin{figure}[h]
	\centering
	\subfloat[]{\includegraphics[width=1cm]{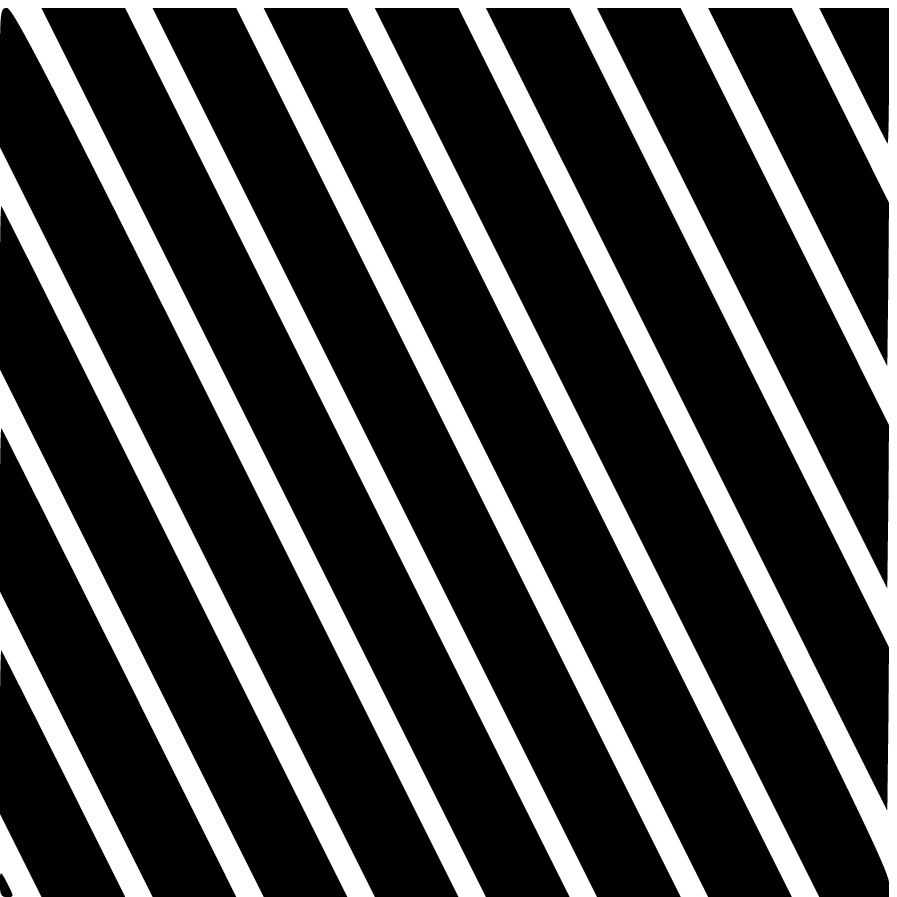}}\quad
	\subfloat[]{\includegraphics[width=1cm]{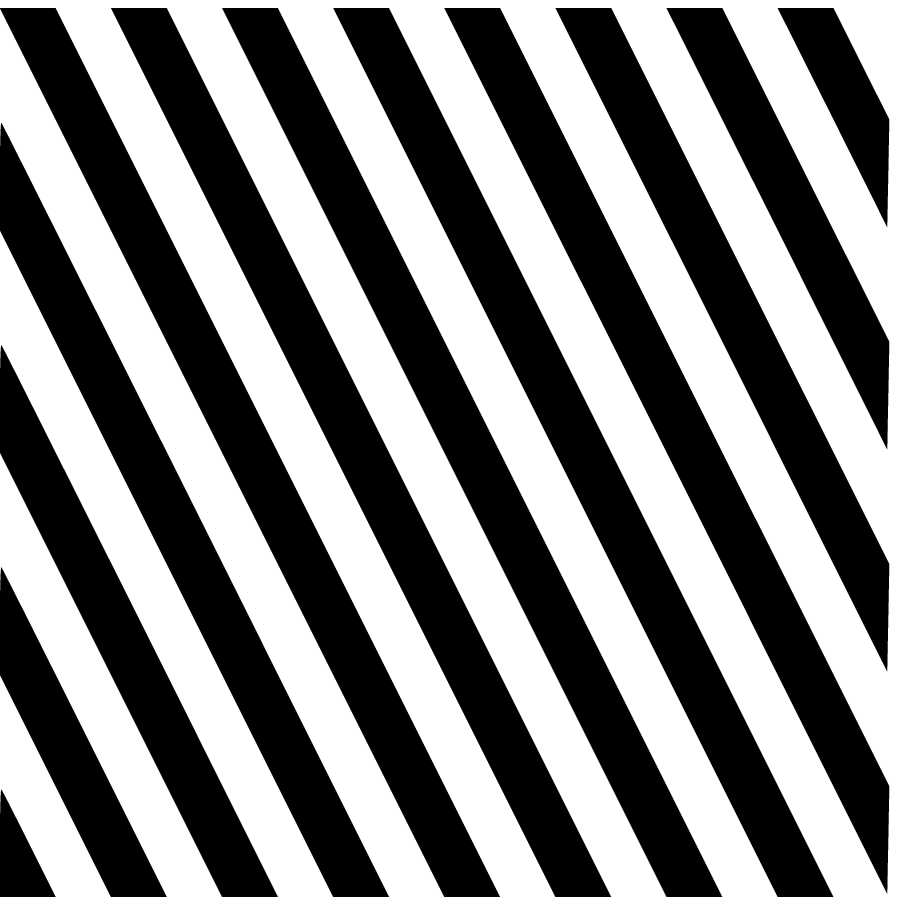}}\quad
	\subfloat[]{\includegraphics[width=1cm]{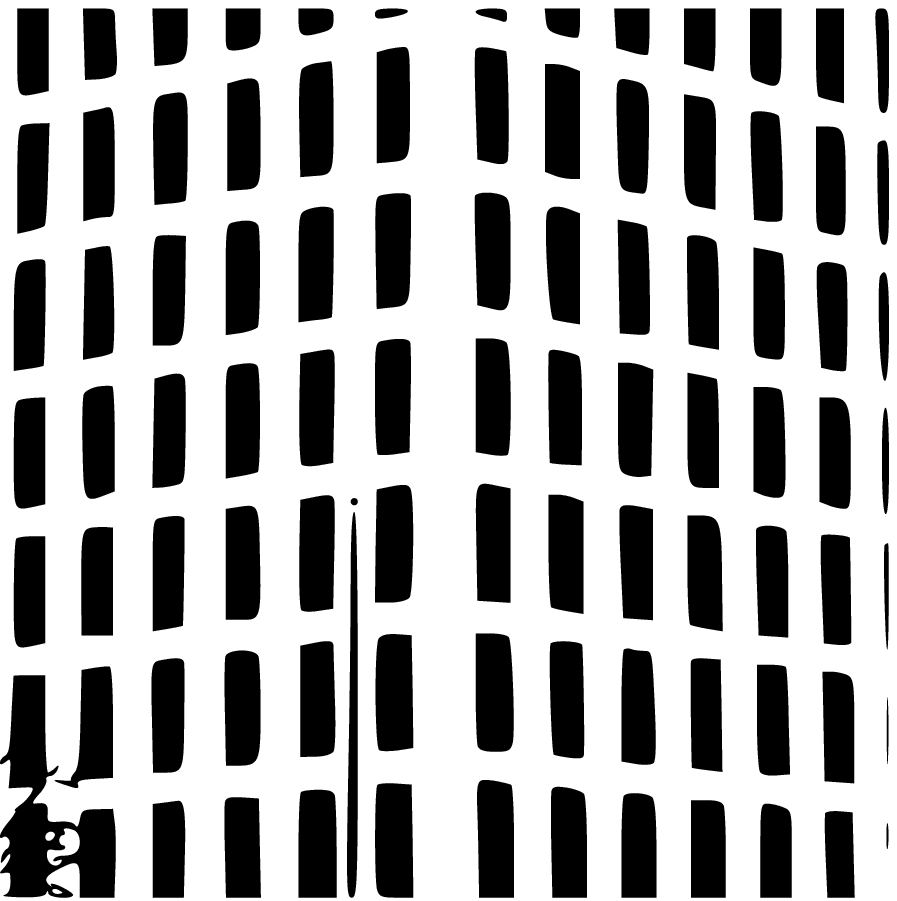}}\quad
	\subfloat[]{\includegraphics[width=1cm]{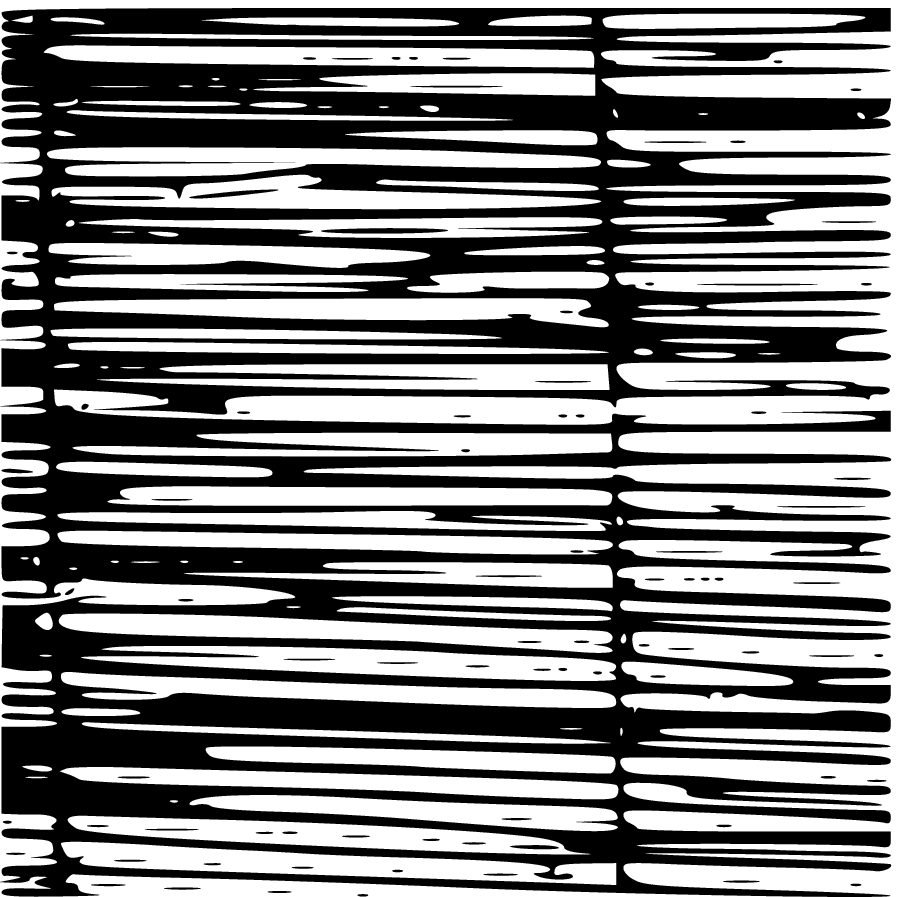}}\quad
	\subfloat[]{\includegraphics[width=1cm]{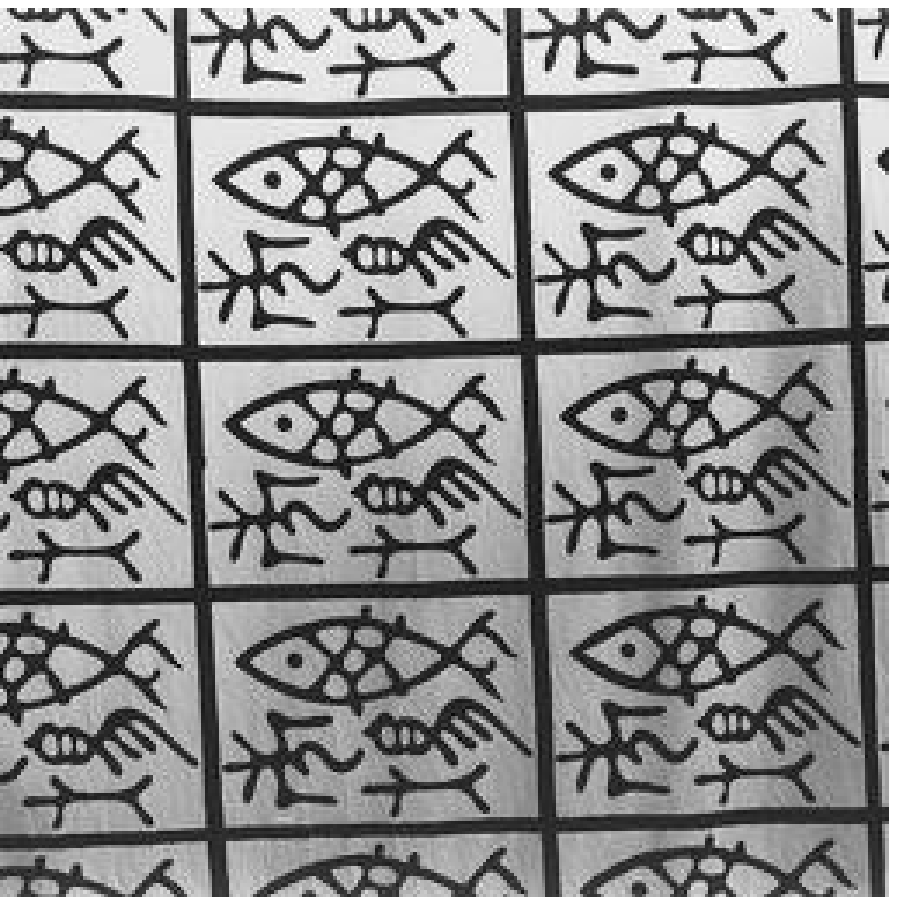}}\quad
	\subfloat[]{\includegraphics[width=1cm]{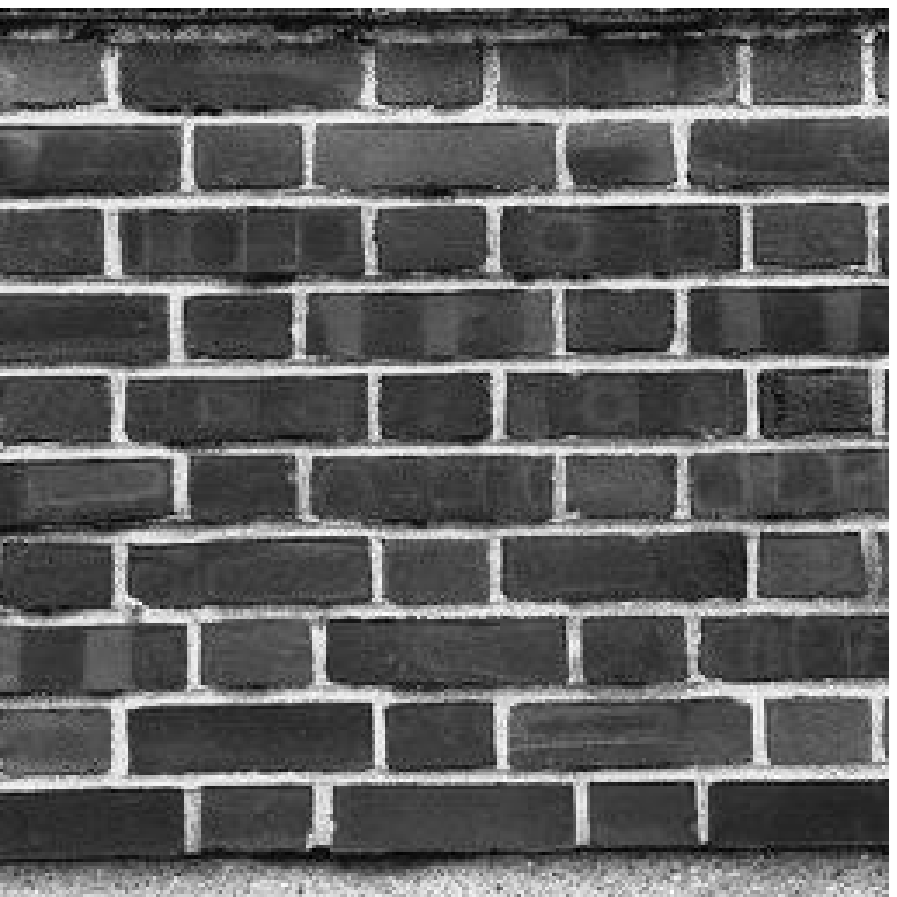}}\\
	\subfloat[]{\includegraphics[width=1cm]{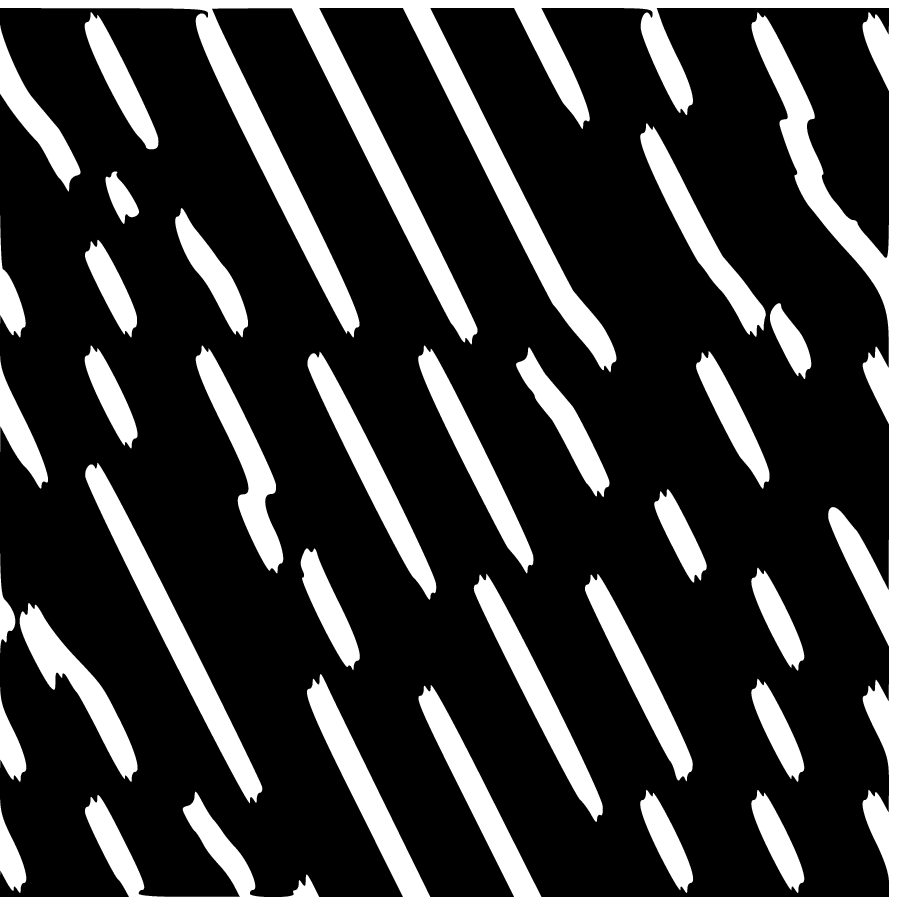}}\quad
	\subfloat[]{\includegraphics[width=1cm]{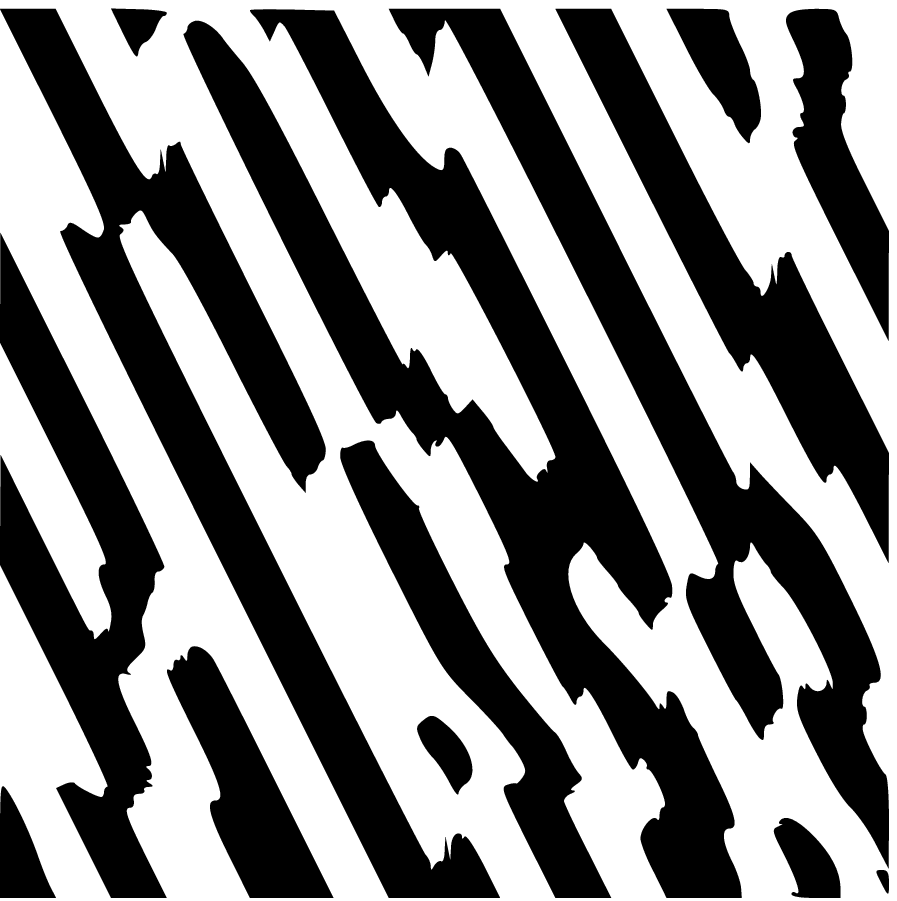}}\quad
	\subfloat[]{\includegraphics[width=1cm]{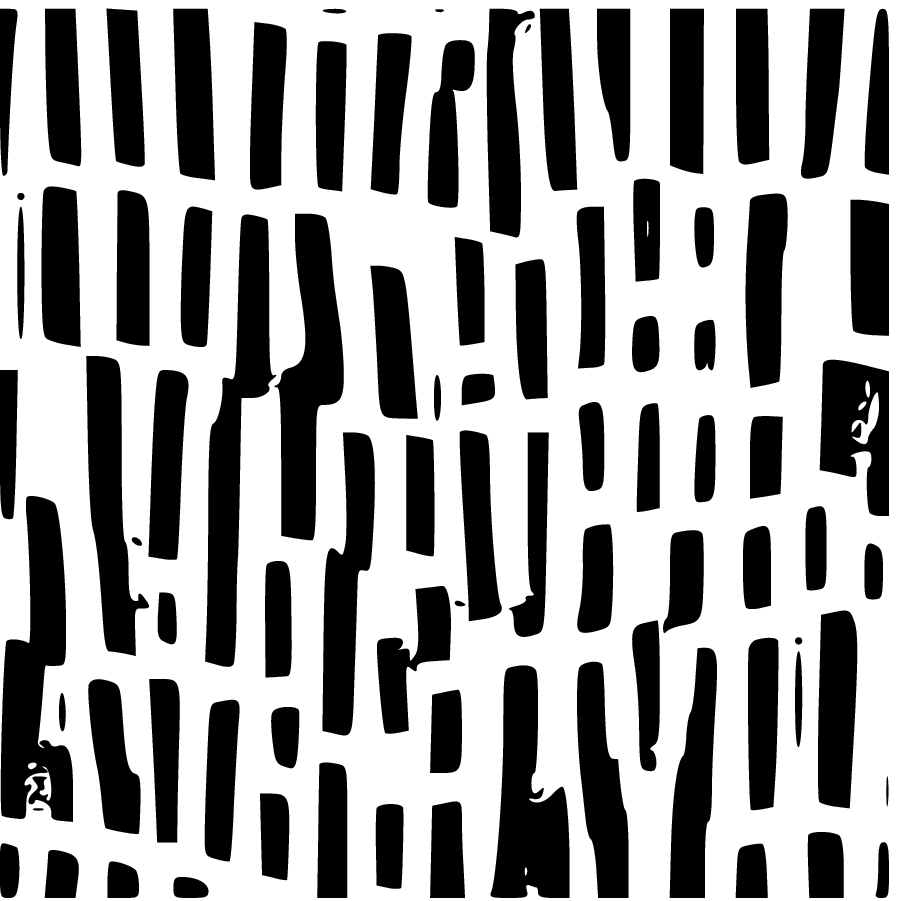}}\quad
	\subfloat[]{\includegraphics[width=1cm]{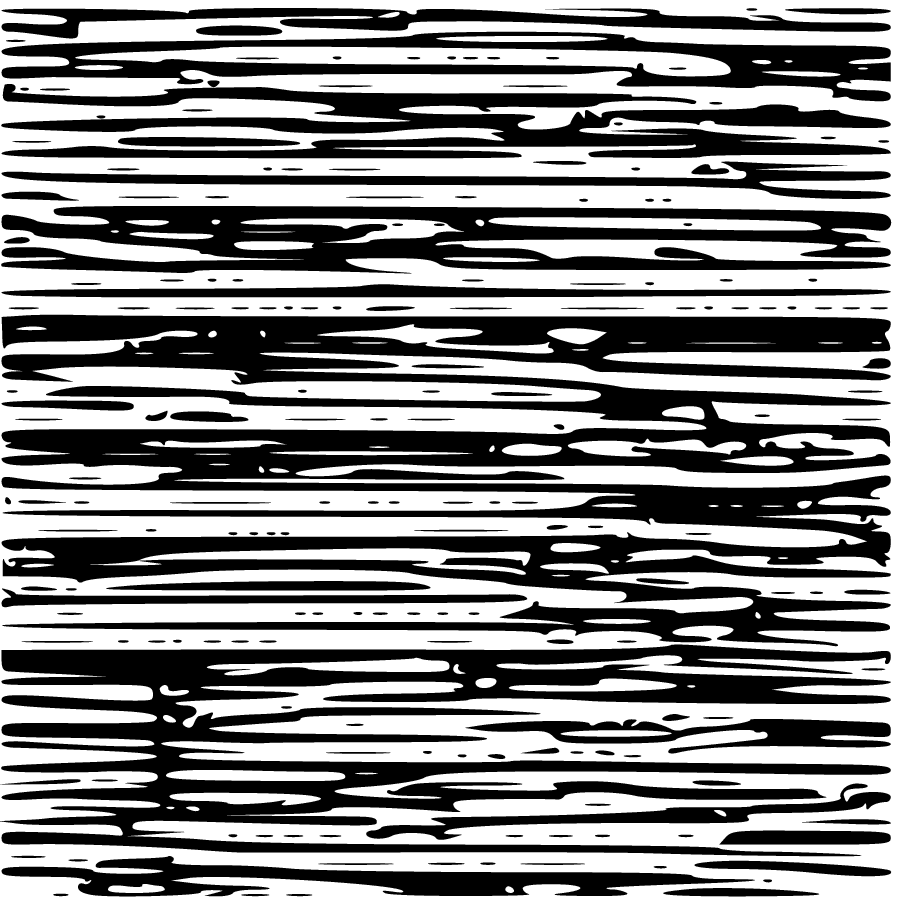}}\quad
	\subfloat[]{\includegraphics[width=1cm]{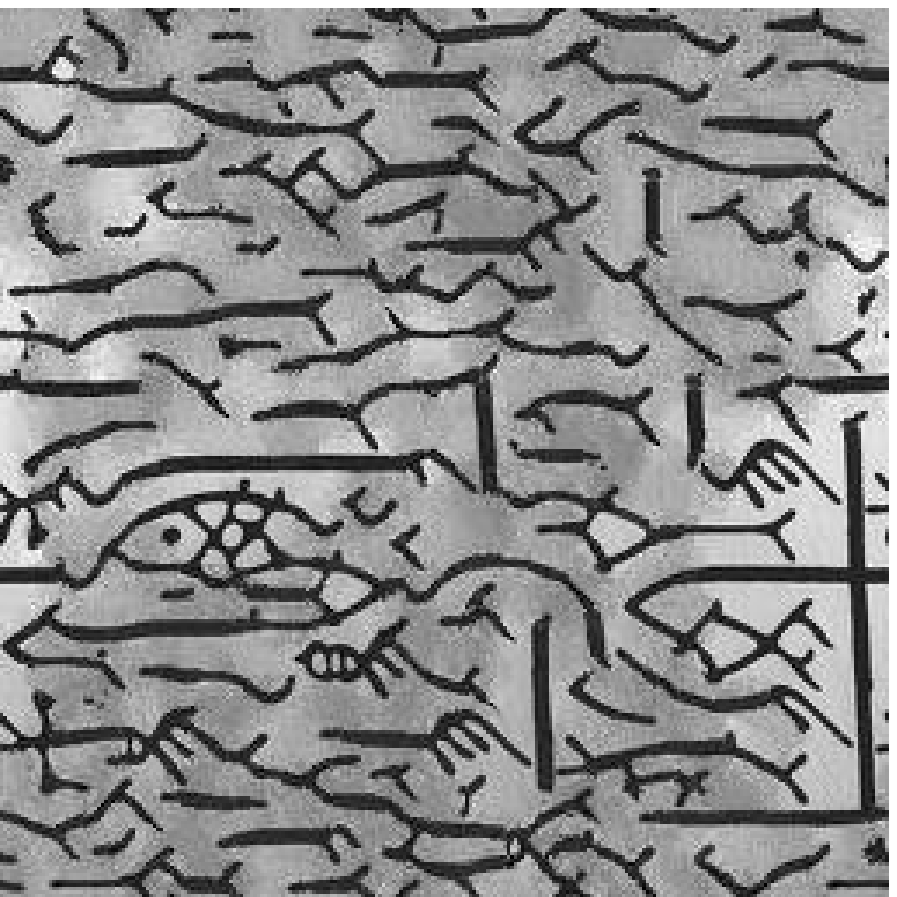}}\quad
	\subfloat[]{\includegraphics[width=1cm]{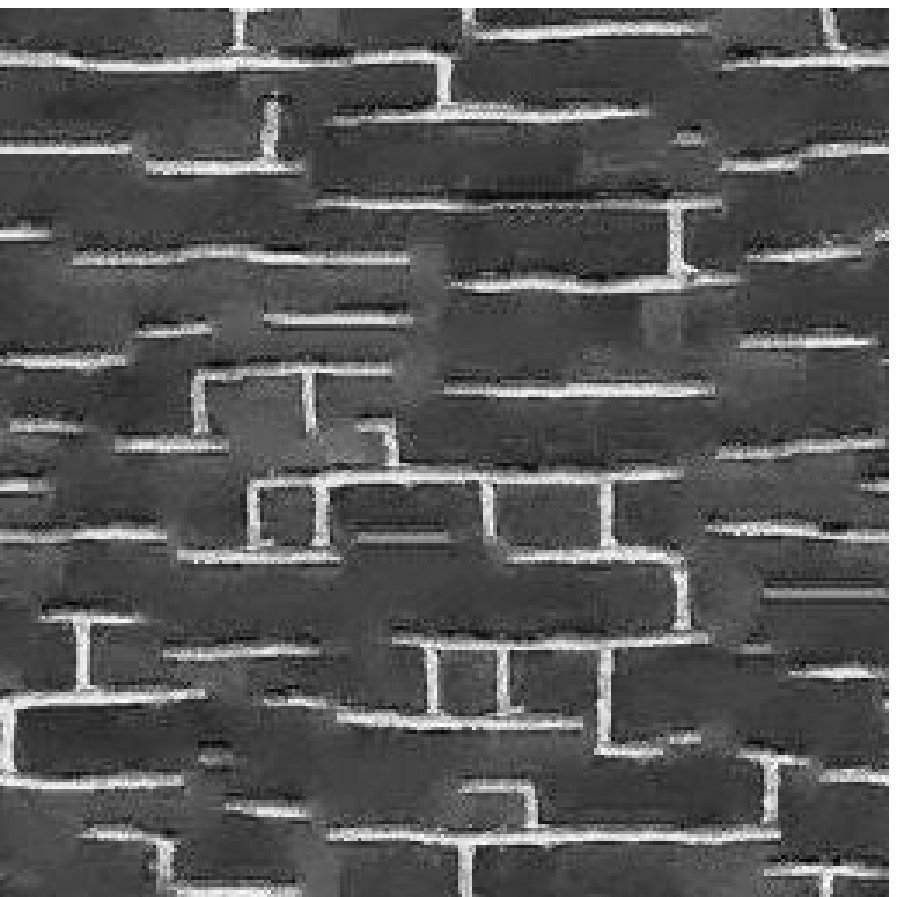}}		
	\caption{Examples of periodic textures (a) to (f) and their corresponding synthetic texture (g) to (l).}
	\label{fig:texture}
\end{figure}

\begin{table}[!t]
\caption{Numerical values of $\PEG$ for periodic textures and synthetic textures in \Fig{texture}.}
\label{table_1}
\centering
\begin{tabular}{|c|c|c|c|c|c|c|}
\hline
Periodic texture & (a) & (b) & (c) & (d) & (e) & (f) \\
\hline
Entropy value & 0.568&0.623&0.328&0.484&0.823&0.842\\
\hline
Synthetic texture & (g) & (h) & (i) & (j) & (k) & (l) \\
\hline
Entropy value &   0.7922& 0.829&0.817&0.852&0.865&0.875\\
\hline
\end{tabular}
\end{table}
Considering the directed graph depicted in Fig.~\ref{sub:directed-graph} and for $m=4$, we compute the $\PEG$. Results in Table~\ref{table_1} shown that the permutation entropy of a synthesized texture are higher than of its corresponding periodic texture. Hence, the algorithm discriminate synthetic periodic from periodic textures in agreement with \cite{Morel2021,Azami2017a}. 

\subsection{Permutation entropy for signals on graphs}
  
Here, the signal $\X=\{x_i\}_{i=1}^N$ will be white Gaussian noise. The entropy has different values depending on the graph underlying, even if $\X$ remains the same.

\subsubsection{Random signals on different graph structures}\label{sub1} We show that the algorithm is able to distinguish different values of regularity degree of the graph.
 
 We considerer several graph structures $\G$ on $N$ vertices with $\X$ as a graph signal: $G_1$ the cycle graph, $G_2$ the complete bipartite graph with bipartition $V_1$ and $V_2$ where $\card{V_1}=\card{V_2}=N/2$, and $G_3$ the complete graph. 
Consider $N=500$, we generate 20 realisations of signal $\X$ and we compute $\PEG$ for $G=G_1, G_2$ and $G_3$ for different embedding values $m$. In \Fig{regular}, we show the mean and standard deviation for dimensions $2\leq m\leq 8$. 
\begin{figure}[h]
	\centering
	\includegraphics[scale=.29]{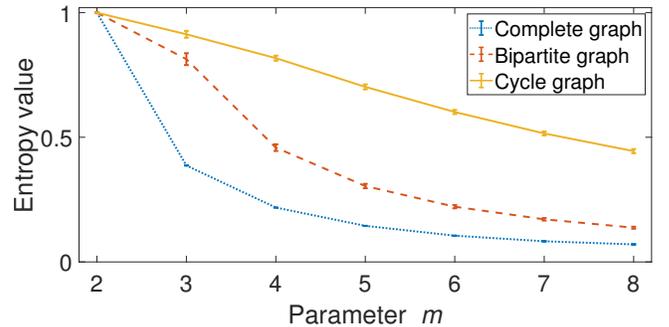}
	\caption{Permutation entropy measures of the random signal $\X$ on several underlying graphs $G$ on $500$ vertices.}
	\label{fig:regular}
\end{figure}

Observe that $G_1$, $G_2$ and $G_3$ are $2, N/2$ and $N-1$ regular graph, respectively. The analysis shows the random signal $\X$ on the cycle graph cycle $\G_1$ has larger entropy values than the same signal on the complete graph $\G_3$. In general, the entropy value for the signal decreases as the degree of regularity of the graph increases. Formally, denote by $G(k,N)$ a $k$-regular graph on $N$ vertices, then:
$\lim_{m\to N}\PE_{G(k,N)}(m) = 0$
and larger value of $k$ increases the convergence ratio. Hence $\PE_{G(k_1,N)}(m)<\PE_{G(k_2,N)}(m)$ for all $m$ and $k_2\ll k_1$. Then, for the same signal, the algorithm is able to detect different degree regularity on the graph structure.

\subsubsection{Erd\H{o}s–Rényi graphs}
The random graph model introduced by Erd\H{o}s–Rényi (ER graphs) is parametrised by the number of vertices $N$, the probability $p$, and it is denoted by $G_{N,p}$. Therefore, ER graphs are models with random connections and they are used to represent common real world data.  

For a fixed $N=2000$, we consider the ER graph $G_{N,p}$ for several values of $p=0.1, 0.3, 0.6, 0.9$. For $20$ realisations of the signal $\X$ we compute its $\PE_{G_{N,p}}$ for $2\leq m \leq 7$. In \Fig{random}, we show the mean and standard deviations. 
\begin{figure}[h]
	\centering
	\includegraphics[scale=.29]{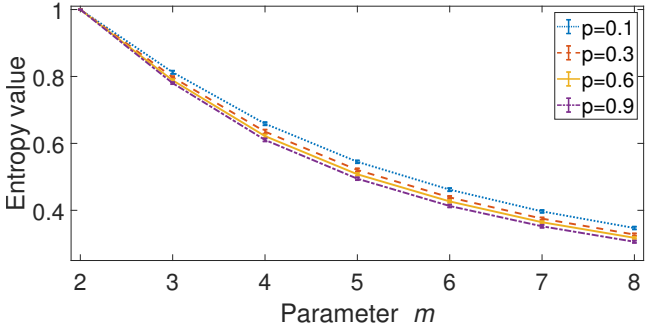}
	\caption{ Erd\H{o}s–Rényi model for values $p$ equal to $0.1, 0.3, 0.6$ and $0.9$ and $N=2000$. Mean value of its $\PEG$ and standard deviation for $20$ simulations.}
	\label{fig:random}
\end{figure}
In all the cases (except for $m=2$) there is not overlapping on the intervals. Then, for a fixed $m>2$ a smaller value of $p$ implies a smaller number of edges (less connectivity) and therefore larger entropy value. Then $\PE_{G_{N,p1}}<\PE_{G_{N,p2}}(m)$ for all $m>2$ and $p2\ll p1$, and the algorithm is able to detects different connectivity degree on the graph structure.

\subsubsection{Controlling the entropy value by changing the graph topology}
Any graph signal can be more regular/irregular depending on the topology graph. Let $\X$ be any signal with $N$ points and consider the embedding dimension $m=2$, for any $\alpha\in[0,1]$ we will are able to construct a graph $\G$ such its entropy of the signal $\X$ is equal (or close enough depending on $N$) to $\alpha$. 

Let $1\leq k\leq N-1$, and without loss of generality, suppose that $x_1,x_2,\dots,x_k$ are the $k$ largest values from the signal $\X$. Consider $G_{k}$ the complete bipartite graph with partition $A=\{1,2,\dots,k\}$ and $B=\{k+1,k+2,\dots,N\}$. In this case,  $PE_{G_k}=- \frac{N-k}{N} \ln ( \frac{N-k}{N})-\frac{k}{N} \ln ( \frac{k}{N})$.

In particular, if $k=1$ then $G_1$ is the star graph with centre on the vertex $1$. The entropy  $\PE_{G_1}=- \frac{N-1}{N} \ln ( \frac{N-1}{N})-\frac{1}{N} \ln ( \frac{1}{N})\rightarrow 0$ as $N\rightarrow\infty$. Then, we have constructed a graph structure with small entropy for the signal $\X$. Similarly, for $N$ even, we consider $k=N/2$, the entropy of the signal $\X$ on the graph $G_{\frac{N}{2}}$ is $-\ln ( \frac{1}{2})$ and its normalised entropy is equal to $1$.

In \Fig{randomalpha}, we show for $N=2000$ the entropy of a signal random signal $\X$ with underlying graph the bipartite graph $G_k$ for $1\leq k\leq 1000$. Observe that permutation entropy for the graph $G_k$ and $G_{N-k}$ are equal because for the graph symmetry. Then, for $m=2$ and any value $\alpha\in[0,1]$ we construct a graph $G_k$ with entropy equal to $\alpha$ (or close enough). Observe that this construction is optimised for $m=2$, and the range of the entropy for larger dimension is narrower. Similar constructions can be done to maximise or minimise the entropy for a fixed embedding dimension $m$.

\begin{figure}[h]
	\centering
	\includegraphics[scale=.29]{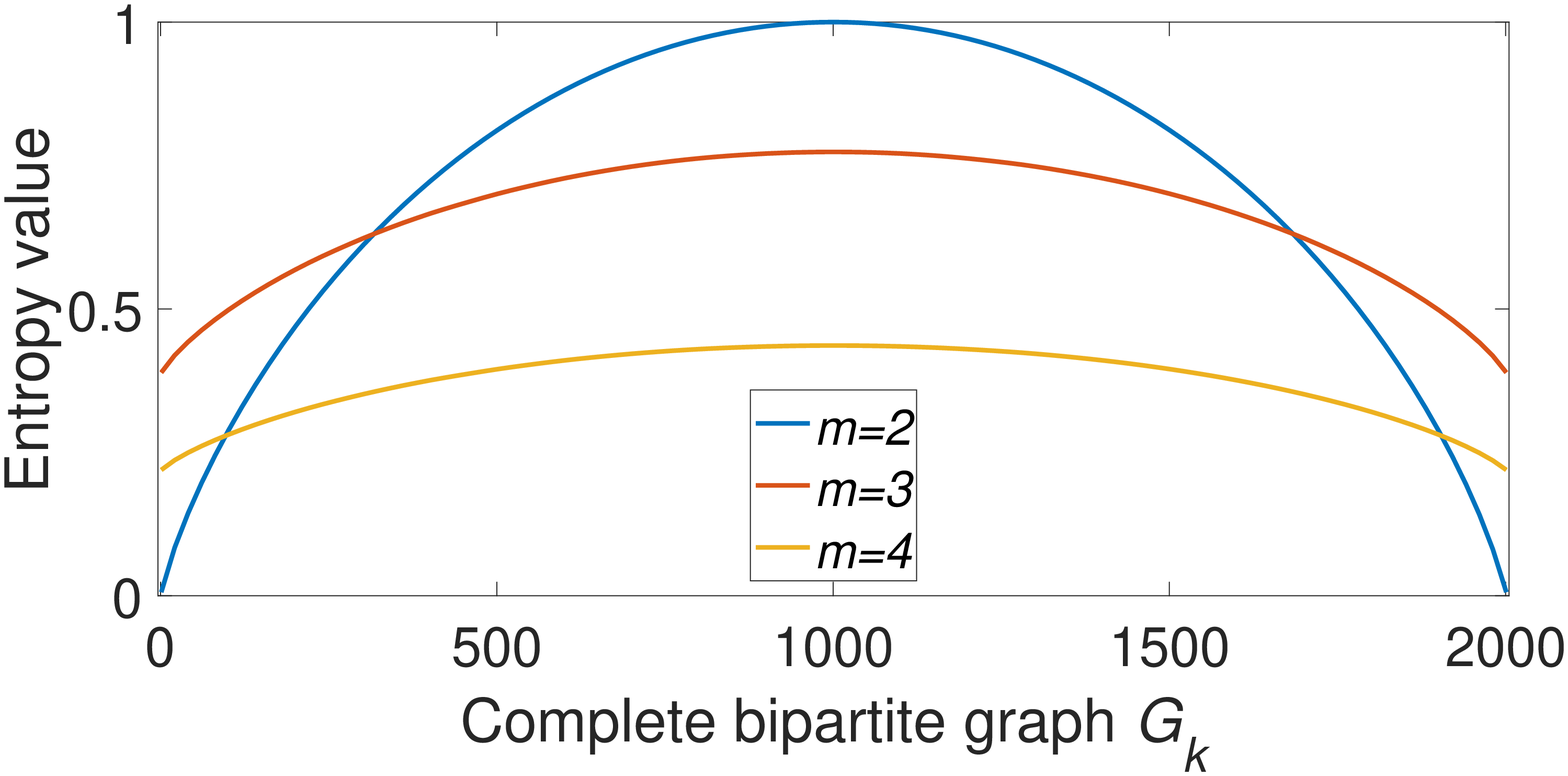}
	\caption{Permutation entropy values of the random signal $\X$ on the complete bipartite graph $G$ on $k$ and $2000-k$ vertices.}
	\label{fig:randomalpha}
\end{figure}

\subsection{A real-world data example: temperature data}
We use the temperature readings of ground stations observed in Brittany for January 2014~\cite{Signal2015}. The graph is defined as follows: each vertex represents the ground station, and the weighted edges between vertices are given by a Gaussian kernel of the Euclidean distance between vertices~\cite{Shuman2016}:

\[   
W_{ij} = 
\begin{cases}
	\exp\left(\frac{-d(i,j)^2}{2\sigma_1^2} \right)  &\quad\text{if } d(i,j)\leq \sigma_2\\
	\text{0} &\quad\text{otherwise.} \\ 
\end{cases}
\]  

Similarly to \cite{Signal2015}, we use $\sigma_1^2=5.1^8$ and $\sigma_2=10^5$. Let $\X_1$ be the signal corresponding to the temperature observation at 14:00 (January 27, 2014) shown in Figure~\ref{sub:mapA} and $\X_2$ the signal at 04:00 (January 23, 2014) shown in Fig.~\ref{sub:mapB}. The signal $\X_1$ is more \emph{irregular} than $\X_2$ in the following sense: In $\X_1$, the maximum/minimum values are more dispersed along the map. In contrast, the extreme values are more grouped in $\X_2$. In Figure \ref{sub:mapB}, we see that the lowest temperatures are localised in the North-East, the highest in the South-West and the transition between ground stations are smother. While in Fig.~\ref{sub:mapA}, the distribution of the temperatures is more random. This fact is captured by the patterns computed for the entropy. 

Formally, the irregularity is measured using $\PEG$ in both signals. For example, consider $m=4$, $L=1$ for compute the entropy, that requires up to $24$ permutation pattern. For each vertex, one pattern is generated (\Eq{eq:average}). In the signal $\X_1$, $13$ patterns appear with distribution $\{1,1,1,2,2,2,3,3,3,3,4,4,8\}$ (see Figure \ref{sub:mapA2}) and for the signal $\X_2$, $5$ patterns are formed with distribution $\{1,1,2,13,20\}$ (see Figure \ref{sub:mapB2}). The $\PEG$ of the signal $\X_1$ is $0.7529$, while for $\X_2$ is $0.3313$.

In general, the entropy values for the signal $\X_1$ are  $0.9995, 0.8729, 0.7529, 0.5438$ for $m=2,3,4,5$ and for $\X_2$ are  $0.9740, 0.5667, 0.3313, 0.2739$ for $m=2,3,4,5$. In all the cases, the entropy for $\X_2$ is smaller than $\X_1$.

In fact, the difference (and relations) of entropy values for 04:00 and 14:00 are preserved for all embedding dimensions $m$. We compute the entropy for all temperature measures (one per hour for 31 days). Table~\ref{table_2} shows the average of the entropy values for the temperature signals at 04:00 and at 14:00. The entropy value of the temperature measurement is higher at 14:00 than at 04:00. We speculate this reflects a more irregular distribution of temperatures over the ground at 14:00 (during the early afternoon when temperatures could reach higher values depending on the geographical situation of each station) than at 4:00 in the middle of the night.
\begin{table}[!h]

	\caption{Average of $\PEG$ values for the temperature measure at 04:00 and 14:00 .}
\label{table_2}
\centering

\begin{tabular}{|c|c|c|c|c|c|c|}
	\hline
	& $m=2$ & $m=3$ & $m=4$ & $m=5$  \\
	\hline
	Temperature at 04:00 & 0.969&0.830&0.572&0.421\\
	\hline
	Temperature at 14:00 & 0.973 & 0.849 & 0.618 & 0.443   \\
	\hline
\end{tabular}
\end{table}

\begin{figure}
	\centering
	\begin{subfigure}[t]{0.5\textwidth}

		\includegraphics[scale=.29]{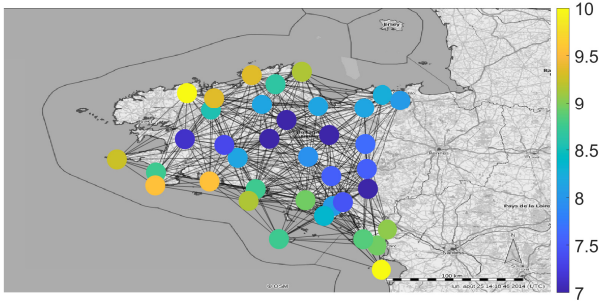}
		\caption{A temperature observation at 14:00 (January 27, 2014), denoted as signal $\X_1$.}
		\label{sub:mapA}
	\end{subfigure}%
\\
	\begin{subfigure}[t]{0.5\textwidth}
		\centering
		\includegraphics[scale=.29]{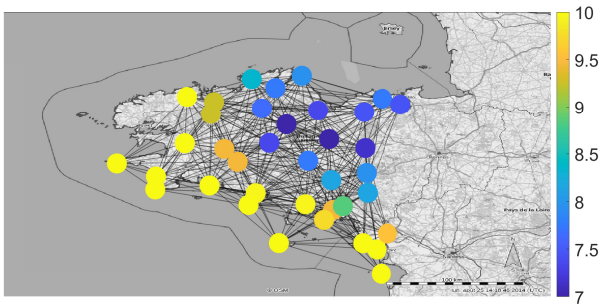}
		
		\caption{A temperature observation at 04:00 (January 23, 2014), denoted as signal $\X_2$.}
		\label{sub:mapB}	
	\end{subfigure}

	\caption{ Two different readings of temperatures in Brittany during January 2014.}
	\label{fig:mapas}
\end{figure}

\begin{figure}
	\centering
	\begin{subfigure}[t]{0.5\textwidth}
		\includegraphics[scale=.29]{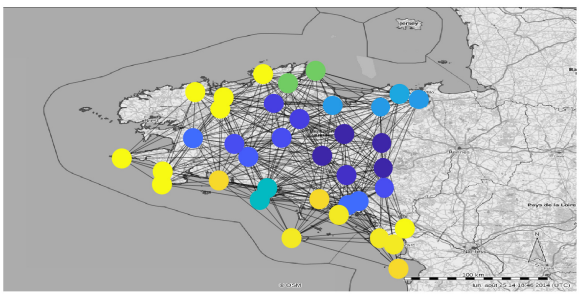}
		\caption{Patterns of the signal $\X_1$ for $m=4$.}
		\label{sub:mapA2}
	\end{subfigure}%
\\
	\begin{subfigure}[t]{0.5\textwidth}
		\centering
		\includegraphics[scale=.29]{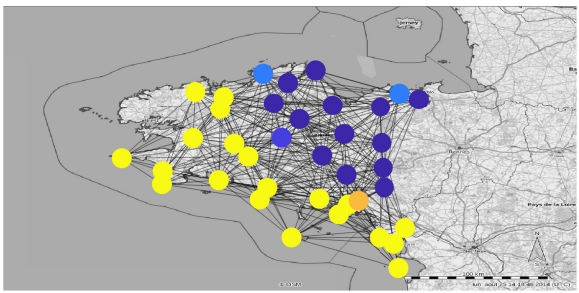}
		
		\caption{Patterns of the signal $\X_2$ for $m=4$.}
		\label{sub:mapB2}	
	\end{subfigure}

	\caption{ Each vertex correspond to one pattern of the signal. Equal colours are equal patterns.}
	\label{fig:mapas2}
\end{figure}	
In this example, we see how the performance of $\PEG$ is different from the smoothness signal definition in \Eq{eq:smooth}. While the smoothness of the $\X_1$ is $221.9$ and for $\X_2$ is $138.2$, i.e. is smoother $\X_1$ than $\X_2$, but visually it seems more \emph{regular} $\X_2$ than $\X_1$. Recall that we are interested in the change of the pattern and the smoothness in the change of values (\Prp{cX} and \Eq{eq:smooth}).

\section{Conclusions and future work}
\label{sec:conc}
In this paper, we generalised the permutation entropy for graph signals. Some modifications and extensions of the classical $\PE$ have been developed in the literature~\cite{Chen2019,Morel2021,azami2019two}. However, we introduce for the first time an entropy measure for signals on general irregular domains defined by graphs.
In particular, we observe that by considering the underlying graph $\G$ as a path (1D), the results of $\PEG$ coincide with the results for standard algorithms on time series (as the original $\PE$~\cite{Bandt2002}). Moreover, our graph algorithm also enables applying $\PE$-related analysis to images (2D).

We also observe that the results depend on how much information we have about the underlying graph. Weights or directions on the edges give different kinds of relations between the signals at neighbouring vertices captured by the algorithm.

We explore how the same signal changes its entropy depending on the topology of the graph and how the same underlying graph with signals with different dynamics has different $\PEG$. It demonstrates the importance of the signal and graph for computing the entropy values.

Some future lines of research are the following:
\begin{itemize}
	\item Extend other one-dimensional entropy metrics to irregular domains (e.g., dispersion entropy).
	\item Generate surrogate graph signals and test the nonlinearlity of the signals defined on the graph.
	\item Study the relationship between properties of the graph (for example, the spectrum of the graph Laplacian), and the regularity of the signal.  This would also be useful to help determine how to define the graph for a given graph signal that would be subject to entropy analysis.
\end{itemize}

We expect the algorithm presented in this paper to enable
the extension of similar techniques that inspect nonlinear dynamics from data acquired over irregular graphs.  

The MATLAB code used in this paper are
freely available at \url{https://github.com/JohnFabila/PEG}.

\ifCLASSOPTIONcaptionsoff
  \newpage
\fi

\end{document}